%% file: thesis.tex
\documentclass[12pt,letterpaper]{report}
\usepackage[latin1]{inputenc}
\usepackage{amsmath}
\usepackage{amsfonts}
\usepackage{setspace}
\title{Title}

\makeatletter
\renewcommand{\ps@plain}{
\renewcommand\@oddhead{\hfill\normalfont\textrm{\thepage}}
\renewcommand\@evenhead{}
\renewcommand\@oddfoot{}
\renewcommand\@evenfoot{}}

\def\@makechapterhead#1{%
  \vspace*{-48\p@}%
  {\parindent \z@ \raggedright \normalfont
    \ifnum \c@secnumdepth >\m@ne
        \huge\bfseries \@chapapp\space \thechapter
        \par\nobreak
        \vskip 0\p@
    \fi
    \interlinepenalty\@M
    \Huge \bfseries #1\par\nobreak
    \vskip 20\p@
  }}

\def\@makeschapterhead#1{%
  \vspace*{-48\p@}%
  {\parindent \z@ \raggedright
    \normalfont
    \interlinepenalty\@M
    \Huge \bfseries  #1\par\nobreak
    \vskip 20\p@
  }}
\makeatother

\renewcommand{\thepage}{\roman{page}}

\usepackage{amssymb}
\usepackage{amsthm}
\usepackage{mathrsfs}
\usepackage{listings}
\usepackage{url}
\usepackage{enumerate}
\usepackage[titletoc]{appendix}
\usepackage{tikz}
\usepackage{amscd}
\usepackage{verbatim}
\usepackage{xspace}
\usepackage{array}
\usepackage[all]{nowidow}

\usetikzlibrary{snakes,arrows,shapes,patterns}

\allowdisplaybreaks

\newcommand{\kk}{\Bbbk}
\newcommand{\NN}{\mathbb N}
\newcommand{\PP}{\mathbb P}

\newcommand{\ZZ}{\mathbb Z}
\newcommand{\stirlingtwo}[2]{\genfrac{\lbrace}{\rbrace}{0pt}{}{#1}{#2}}
\newcommand{\define}[1]{\textbf{#1}}
\newcommand{\statement}[1]{\mathfrak{#1}}
\newcommand{\sheaf}[1]{\mathscr{#1}}
\newcommand{\defect}{\delta}
\newcommand{\tuple}[1]{\textup{\textbf{#1}}}
\newcommand{\nth}{\textsuperscript{th}\xspace}

\DeclareMathOperator{\expdim}{expdim}
\DeclareMathOperator{\Split}{Split}
\DeclareMathOperator{\Span}{span}
\DeclareMathOperator{\rank}{rank}
\DeclareMathOperator{\Seg}{Seg}

\newtheorem{theorem}{Theorem}[section]
\newtheorem{proposition}[theorem]{Proposition}
\newtheorem{lemma}[theorem]{Lemma}
\newtheorem{corollary}[theorem]{Corollary}
\newtheorem{conjecture}[theorem]{Conjecture}
\newtheorem{problem}[theorem]{Problem}
\theoremstyle{definition}
\newtheorem{definition}[theorem]{Definition}
\newtheorem{computation}[theorem]{Macaulay2 Computation}

    \newtheoremstyle{TheoremNum}
        {\topsep}{\topsep}              
        {\itshape}                      
        {}                              
        {\bfseries}                     
        {.}                             
        { }                             
        {\thmname{#1}\thmnote{ \bfseries #3}}
    \theoremstyle{TheoremNum}
    \newtheorem{thmn}{Theorem}

\begin{document}
\pagestyle{myheadings}

\thispagestyle{empty}

\begin{center}
Nondefective secant varieties of varieties of completely decomposable forms\\
A Dissertation\\
Presented in Partial Fulfillment of the Requirements for the\\
Degree of Doctor of Philosophy\\
with a\\
Major in Mathematics\\
in the\\
College of Graduate Studies\\
University of Idaho\\
\vspace{84pt}
by\\
Douglas A. Torrance\\
\vspace{48pt}
June 2013\\
\vspace{60pt} 
Major Professor: Hirotachi Abo, Ph.D.\\
\end{center}
\pagebreak

\addcontentsline{toc}{chapter}{Authorization to Submit Dissertation}
\section*{\begin{center}AUTHORIZATION TO SUBMIT DISSERTATION\end{center}}
\begin{flushleft}
This dissertation of Douglas A. Torrance, submitted for the degree of Doctor of Philosophy with a major in Mathematics and titled ``Nondefective secant varieties of varieties of completely decomposable forms,'' has been reviewed in final form. Permission, as indicated by the signatures and dates given below, is now granted to submit final copies to the College of Graduate Studies for approval.
\end{flushleft}
\begin{singlespace}
\ \ \ \ \ Major Professor\indent\underline{\makebox[2.8in][l]{\ }}Date\underline{\makebox[1.2in][l]{\ }}\\
\ \ \indent\indent\indent\indent\indent\indent\indent Hirotachi Abo\\
\ \\
\ \ \ \indent Committee\\
\ \ \ \indent Members\indent\indent\ \ \ \ \ \ \underline{\makebox[2.8in][l]{\ }}Date\underline{\makebox[1.2in][l]{\ }}\\
\ \ \indent\indent\indent\indent\indent\indent\indent Roger Cole\\
\ \\
\ \ \indent\indent\indent\indent\indent\indent\ \ \ \underline{\makebox[2.8in][l]{\ }}Date\underline{\makebox[1.2in][l]{\ }}\\
\ \ \indent\indent\indent\indent\indent\indent\indent Jennifer Johnson-Leung\\
\ \\
\ \ \indent\indent\indent\indent\indent\indent\ \ \ \underline{\makebox[2.8in][l]{\ }}Date\underline{\makebox[1.2in][l]{\ }}\\
\ \ \indent\indent\indent\indent\indent\indent\indent Alexander Woo\\
\ \\
\ \ \indent Department\\
\ \ \indent Administrator\indent\ \ \ \underline{\makebox[2.8in][l]{\ }}Date\underline{\makebox[1.2in][l]{\ }}\\
\ \ \indent\indent\indent\indent\indent\indent\indent Monte Boisen\\
\ \\
\ \ \indent Discipline's\\
\ \ \indent College Dean\indent\ \ \ \ \ \underline{\makebox[2.8in][l]{\ }}Date\underline{\makebox[1.2in][l]{\ }}\\
\ \ \indent\indent\indent\indent\indent\indent\indent Paul Joyce\\
\ \\
Final Approval and Acceptance by the College of Graduate Studies\\
\ \\
\ \ \indent\indent\indent\indent\indent\indent\ \ \underline{\makebox[2.8in][l]{\ }}Date\underline{\makebox[1.2in][l]{\ }}\\
\ \ \indent\indent\indent\indent\indent\indent\indent Jie Chen\\
\end{singlespace}
\pagebreak

\addcontentsline{toc}{chapter}{Abstract}
\chapter*{Abstract}
A variation of Waring's problem from classical number theory is the question, ``What is the smallest number $s$ such that any generic homogeneous polynomial of degree $d$ in $n+1$ variables may be written as the sum of at most $s$ products of linear forms?''  This question may be answered geometrically by determining the smallest $s$ such that the $s$\nth secant variety of the variety of completely decomposable forms fills the ambient space.  If this secant variety has the expected dimension, it is called nondefective, and $s=\left\lceil\binom{n+d}{d}/(dn+1)\right\rceil$.  It is conjectured that the secant variety is always nondefective unless $d=2$ and $2\leq s\leq\frac{n}{2}$.  We prove several special cases of this conjecture.  In particular, we define functions $s_1$ and $s_2$ such that the secant variety is nondefective when $n\geq 3$ and $s\leq s_1(d)$ or when $n=3$ and $s\geq s_2(d)$ and a function $c$ such that the secant variety is nondefective when $d\geq n\geq 4$ and $s\leq 2^{n-3}c(n,d)$.  We further show that the secant variety is nondefective when $s\leq 30$ unless $d=2$ and $2\leq s\leq\frac{n}{2}$.

\pagebreak

\addcontentsline{toc}{chapter}{Acknowledgements}
\chapter*{Acknowledgements}

First and foremost, I would like to thank my advisor, Hirotachi Abo.  His guidance, helpfulness, and patience have been remarkable.  I would also like to thank the rest of the faculty in the Department of Mathematics at the University of Idaho, especially my committee members Alexander Woo and Jennifer Johnson-Leung for their insight and comments, and department chair Monte Boisen for his cheer and support.  I also thank my fellow graduate student Jia Wan for many helpful discussions.

I would also like to thank my family.  Thanks go to my father, Douglas E. Torrance, who inspired me to pursue a career in academia, and to the memory of my mother, Diane Palliser.  I thank my infant son Gabriel, whose smiles light up my day, and my mother-in-law, Bobbi Cady, for all the help she has been this past year.  Finally, I thank my wife Sarah, my partner through all of this.  I love you!

\pagebreak

\addcontentsline{toc}{chapter}{Dedication}
\topskip0pt
\vspace*{\fill}
\begin{center}
\textit{To Sarah and Gabe.}
\end{center}
\vspace*{\fill}
\pagebreak

\addcontentsline{toc}{chapter}{Table of Contents}
\tableofcontents
\pagebreak

\addcontentsline{toc}{chapter}{List of Tables}
\listoftables
\pagebreak

\addcontentsline{toc}{chapter}{List of Figures}
\listoffigures
\pagebreak

\setcounter{page}{1}
\renewcommand{\thepage}{\arabic{page}}
\chapter{Preliminaries}

\section{Introduction}

Consider the following classical problem from number theory.

\begin{problem}[Waring's problem]
What is the smallest $s$ such that any natural number may be written as the sum of at most $s$ $d$\nth powers of natural numbers?
\end{problem}

This problem was named after Edward Waring, who in 1770, stated the following theorem without proof \cite[Theorem 47]{Waring}.

\begin{theorem}
There exists a function $s:\NN\rightarrow\NN$ such that, for every $d\in\NN$, any natural number may be minimally written as the sum of at most $s(d)$ $d$\nth powers.  In particular, $s(2)=4$, $s(3)=9$, and $s(4)=19$.
\end{theorem}

The existence of $s$ was not proven until 1909 by Hilbert \cite{Hilbert}.  Lagrange proved that $s(2)=4$ in 1770 \cite{Lagrange} in what has become known as ``Lagrange's four-square theorem.''  Wieferich proved much of the $d=3$ case in 1908 \cite{Wieferich}, with gaps filled in by Kempner in 1912 \cite{Kempner}, and Balasubramanian, Deshouillers, and Dress finally proved that $s(4)=19$ in 1986 \cite{BDD}.

Waring's problem serves as an inspiration for similar problems, such as the following Waring's problem for polynomials.

\begin{problem}[Waring's problem for polynomials]\label{WaringPolynomials}
What is the smallest $s$ such that any generic polynomial in $n+1$ variables of degree $d$ may be written as the sum of at most $s$ $d$\nth powers of linear forms?
\end{problem}

Waring's problem for polynomials was solved by Alexander and Hirschowitz, as will be discussed in Section \ref{SecantVarieties}.

In this dissertation, we will examine a generalization of Waring's problem for polynomials.  A \define{completely decomposable form} is a homogeneous polynomial which is the product of linear forms.

\begin{problem}[Waring's problem for completely decomposable forms]\label{WaringCDF}
What is the smallest $s$ such that any generic polynomial in $n+1$ variables of degree $d$ may be written as the sum of at most $s$ completely decomposable forms?
\end{problem}

First, note that such an $s$ exists, as every polynomial is the sum of monomials, and every monomial is a completely decomposable form.  Since there are, up to scalar multiplication, $\binom{n+d}{d}$ monomials of degree $d$ in $n+1$ variables, we have $s\leq\binom{n+d}{d}$.

Our goal is to determine the exact value of this $s$ for as many pairs $(n,d)$ as possible.  Our main tool will be secant varieties of varieties of completely decomposable forms, denoted $\sigma_s(\Split_d(\PP^n))$, which we will introduce more precisely later in this chapter. 

This version of Waring's problem is easy to answer if the secant variety is known to be nondefective, \textit{i.e.}, it has the expected dimension.  It is known that $\sigma_s(\Split_d(\PP^n))$ is nondefective if $d=1$, $d=2$ and $s=1$ or $n<2s$, $n\leq 2$, or $3(s-1)\leq n$.  Also, there exist both upper and lower bounds for $s$ such that $\sigma_s(\Split_3(\PP^n))$ and $\sigma_s(\Split_d(\PP^3))$ are known to be nondefective.  See Section \ref{KnownResults} for more details on these results.

We prove the three additional results regarding the nondefectivity of particular cases.

In the first result, we use a method we call ``restriction induction'' to find upper and lower bounds for $s$ for which $\sigma_s(\Split_d(\PP^3))$ is nondefective.  These bounds, which are proven in Theorem \ref{nEqualsThreeResult}, are better than the bounds which were previously known.  We later use a method we call ``splitting induction'' to extend the upper bound to $n\geq 3$ in Corollary \ref{S1Generalized}.

\begin{theorem}\label{S1AndS2}
Let 
\begin{align*}
s_1(d) =
\begin{cases}
\frac{1}{18}d^2+\frac{5}{18}d&\text{ if }d\equiv 0,4\pmod{9}\\
\frac{1}{18}d^2+\frac{5}{18}d+\frac{2}{9}&\text{ if }d\equiv 2,5,8\pmod{9}\\
\frac{1}{18}d^2+\frac{5}{18}d+\frac{2}{3}&\text{ if }d\equiv 1,3\pmod{9}\\
\frac{1}{18}d^2+\frac{5}{18}d+\frac{1}{3}&\text{ if }d\equiv 6,7\pmod{9}\\
\end{cases}\\
\intertext{and}
s_2(d) =
\begin{cases}
\frac{1}{18}d^2+\frac{1}{3}d+1 &\text{ if }d\equiv 0\pmod{6}\\
\frac{1}{18}d^2+\frac{1}{3}d+\frac{1}{2} &\text{ if }d\equiv 3\pmod{6}\\
\frac{1}{18}d^2+\frac{7}{18}d+\frac{5}{9}&\text{ if }d\equiv 1,4,7\pmod{9}\\
\frac{1}{18}d^2+\frac{7}{18}d+1&\text{ if }d\equiv 2\pmod{9}\\
\frac{1}{18}d^2+\frac{7}{18}d+\frac{2}{3}&\text{ if }d\equiv 5\pmod{9}\\
\frac{1}{18}d^2+\frac{7}{18}d+\frac{1}{3}&\text{ if }d\equiv 8\pmod{9}.\\
\end{cases}
\end{align*}
If $s\leq s_1(d)$, then $\sigma_s(\Split_d(\PP^n)$ is nondefective for all $d\in\NN$ and $n\geq 3$.  Ifs $s\geq s_2(d)$, then $\sigma_s(\Split_d(\PP^3))$ is nondefective for all $d\in\NN$. 
\end{theorem}

In the second result, we combine restriction and splitting induction to get an upper bound on $s$ for which $\sigma_s(\Split_d(\PP^n))$ is nondefective for $n\geq 4$.  This result is proven in Section \ref{NAtLeastFour}.

\begin{theorem}\label{ExponentialBound}
Consider the function
\begin{equation*}
c(n,d)=\min\left\{\left\lfloor\frac{\tilde s(d-m)}{g_n(m)}\right\rfloor:0\leq m\leq n-2\right\}
\end{equation*}
where
\begin{equation*}
\tilde s(d) = \begin{cases}
\frac{1}{24}d^2+\frac{1}{12}d&\text{ if }d\equiv 0,4\pmod{6}\\
\frac{1}{24}d^2+\frac{1}{6}d-\frac{5}{24}&\text{ if }d\equiv 1\pmod{6}\\
\frac{1}{24}d^2+\frac{1}{12}d-\frac{1}{3}&\text{ if }d\equiv 2\pmod{6}\\
\frac{1}{24}d^2+\frac{1}{6}d+\frac{1}{8}&\text{ if }d\equiv 3,5\pmod{6}\\
\end{cases}
\end{equation*}
and
\begin{equation*}
g_n(m)=\begin{cases}
n-3 &\text{ if }m=0\text{ or }m=n-3\\
n-4 &\text{ if }n\geq 5\text{ and }m=1\\
1 &\text{ if }m=n-2\\
m(n-m-3) &\text{ if }2\leq m\leq n-4.\\
\end{cases}.
\end{equation*}
If $d\geq n\geq 4$ and $s\leq 2^{n-3}c(n,d)$, then $\sigma_s(\Split_d(\PP^n))$ is nondefective.
\end{theorem}

In the third result, we note that for any fixed $s$, the defectivity or nondefectivity of $\sigma_s(\Split_d(\PP^n))$ is known for all but finitely many cases.  We use computational methods to prove the nondefectivity of these remaining cases for small $s$.  This result is proven in Section \ref{SmallS}.

\begin{theorem}\label{SUpperBound}
If $s\leq 30$, then $\sigma_s(\Split_d(\PP^n))$ is nondefective for all $n,d\in\NN$ unless $d=2$ and $2\leq s\leq\frac{n}{2}$.
\end{theorem}

\pagebreak

\section{Notation}

In this dissertation, $\kk$ is an algebraically closed field of characteristic zero and $R=\kk[x_0,\ldots,x_n]$, \textit{i.e.}, the ring of polynomials in $n+1$ variables with coefficients in $\kk$.  The vector space over $\kk$ of all degree $d$ forms in $R$ is denoted by $R_d$.

For any finite-dimensional vector space $V$ over $\kk$, $\PP V$ is the projective space of lines through the origin in $V$.  If $\dim V=n+1$, then $\PP V$ may also be denoted $\PP^n$.  For any nonzero vector $v\in V$, the line spanned by $v$ is denoted $[v]$ when considered as a point in $\PP^n$.

Consider $s$ varieties $X_1,\ldots,X_s\subset\PP^N$.  Their \define{linear span} $\langle X_1,\ldots,X_s\rangle$ is the smallest linear subspace containing each of them.

Consider the variety $\Sigma$ parameterizing all collections of $s$ points in $\PP^N$.  To say that property $P$ holds for $s$ \define{generic} points in $\PP^N$ is to say the subset of $\Sigma$ corresponding to points for which $P$ holds contains a Zariski open dense subset of $\Sigma$.  (See \cite[Lecture 5]{Harris}.)

Recall that there is a natural map $\pi:\mathbb A^{n+1}\setminus\{0\}\rightarrow\PP^n$ defined by $x\mapsto[x]$.  The \define{affine cone} of a projective variety $X\subset\PP^n$ is the affine variety $\widehat X=\pi^{-1}(X)\cup\{0\}$.  Note that $\dim X = \dim\widehat X-1$.

If $V$ is a vector space with dual $V^*$ and $U$ is a subspace of $V$, then $U^{\perp}$ is the annihilator of $U$, \textit{i.e.}, $U^\perp = \{\varphi\in V^*:\varphi(u)=0\text{ for all }u\in U\}$.

If $V$ is a vector space over $\kk$ with basis $\{v_0,\ldots,v_n\}$, then its $d$\nth \define{symmetric power} is the vector space $S_dV=(\kk[v_0,\ldots,v_n])_d$.

If $\sheaf F$ is a coherent sheaf on $\PP^N$, then each cohomology group $H^i(\PP^n,\sheaf F)$ is vector space over $\kk$ of dimension $h^i(\PP^n,\sheaf F)$.  (See, for example, \cite[Section III.5]{Hartshorne}.)

\pagebreak

\section{Secant Varieties}\label{SecantVarieties}

In this section, we introduce secant varieties and discuss how the result of Alexander and Hirschowitz about secant varieties of Veronese varieties can be used to answer Waring's problem for polynomials (Problem \ref{WaringPolynomials}).

\begin{definition}
Consider a projective variety $X\subset\PP^N$ for some $N\in\NN$.  For some $s\in\NN$, choose $s$ generic points $p_1,\ldots,p_s\in X$.  Then $\langle p_1,\ldots,p_s\rangle$ is an $(s-1)$-plane, \textit{i.e.}, a linear subspace of dimension $s-1$.  The Zariski closure of the union of all the $(s-1)$-planes obtained in this manner is the $s$\nth \define{secant variety} to $X$, denoted $\sigma_s(X)$.
\end{definition}

\begin{lemma}
If $X\subset\PP^N$ is a nonsingular projective variety, then 
\begin{equation*}
\dim\sigma_s(X)\leq\min\{s(\dim X+1)-1,N\}.
\end{equation*}
\end{lemma}
\begin{proof}
Let $\Gamma=\overline{\left\{(p_1,\ldots,p_s,q)\in\left(\left(\prod_{i=1}^sX\right)\setminus\Delta\right)\times\PP^N:q\in\langle p_1,\ldots,p_s\rangle\right\}}$, where $\Delta=\{(x,\ldots,x)\in\prod_{i=1}^s:x\in X\}$.  Then $\dim\Gamma=s\dim X+s-1=s(\dim X+1)-1$.

Let $\pi:\prod_{i=1}^sX\times\PP^N\rightarrow\PP^N$ be the natural projection onto $\PP^N$.  Then, by definition, $\pi(\Gamma)=\sigma_s(X)$.  Consequently, $\dim\sigma_s(X)\leq\min\{\dim\Gamma,N\}$.
\end{proof}

With this in mind, we introduce the following definitions.
\begin{definition}
The \define{expected dimension} of a secant variety to a nonsingular projective variety $X\subset\PP^N$ is 
\begin{equation*}
\expdim\sigma_s(X)=\min\{s(\dim X+1)-1,N\}.
\end{equation*}
A secant variety $\sigma_s(X)$ is \define{defective} if $\dim\sigma_s(X)<\expdim\sigma_s(X)$ and \define{nondefective} otherwise.  The \define{defect} of $\sigma_s(X)$ is 
\begin{equation*}
\defect(\sigma_s(X))=\expdim\sigma_s(X)-\dim\sigma_s(X).
\end{equation*}
\end{definition}

For each $d$, we define the \define{Veronese map} $\nu_d:\PP^n\rightarrow \PP^{\binom{n+d}{d}-1}$ by $[\ell]\mapsto[\ell^d]$.  Each image $\nu_d(\PP^n)$ is a \define{Veronese variety}.

Note that $\sigma_s(\nu_d(\PP^n))$ contains the equivalence classes of all homogeneous polynomials of degree $d$ which can be written as the sum of $s$ $d$\nth powers of linear forms.  In other words, we may solve Waring's problem for polynomials (Problem \ref{WaringPolynomials}) by answering the question, ``What is the smallest $s$ such that $\sigma_s(\nu_d(\PP^n))=\PP^{\binom{n+d}{d}-1}$?''

If $\sigma_s(\nu_d(\PP^n))$ is nondefective for the appropriate values of $s$, then this is trivial.  Indeed, since $\dim\nu_d(\PP^n)=n$, we have $\dim\sigma_s(\nu_d(\PP^n))=\min\{s(n+1),\binom{n+d}{d}\}-1$.  In order for $\sigma_s(\nu_d(\PP^n))=\PP^{\binom{n+d}{d}-1}$, we must have
\begin{align*}
s(n+1)&\geq\binom{n+d}{d}\\
s&\geq\left\lceil\frac{\binom{n+d}{d}}{n+1}\right\rceil.
\end{align*}

It remains to describe all of the defective cases.  The following theorem was conjectured by Palatini in 1903 \cite{Palatini} but not proven until 1995 by Alexander and Hirschowitz \cite{AlexanderHirschowitz}.

\begin{theorem}[Alexander-Hirschowitz theorem]\label{AlexanderHirschowitzTheorem}
The secant variety $\sigma_s(\nu_d(\PP^n))$ is nondefective except for the cases in Table \ref{DefectiveVeronese}.
\end{theorem}

\begin{table}[h]
\centering
\begin{tabular}{|c|c|c|}
\hline
$n$ & $d$ & $s$ \\
\hline
$\geq 2$ & 2 & $2,\ldots,n$\\
2 & 4 & 5 \\
3 & 4 & 9 \\
4 & 3 & 7 \\
4 & 4 & 14 \\
\hline
\end{tabular}
\caption{Defective cases of $\sigma_s(\nu_d(\PP^n))$}
\label{DefectiveVeronese}
\end{table}

For a simplified proof and historical overview of the Alexander-Hirschowitz theorem, see \cite{BrambillaOttaviani}.

In the quadratic case, $n\geq\left\lceil\frac{\binom{n+2}{2}}{n+1}\right\rceil$ for all $n\geq 2$, so the secant variety will fill the ambient space when $s=n+1$.  Note that $1+1=\left\lceil\frac{\binom{1+2}{2}}{1+1}\right\rceil
$, so this also holds for $n=1$.

In the final four cases, $s=\left\lceil\frac{\binom{n+d}{d}}{n+1}\right\rceil$, so the secant variety will fill the ambient space when $s$ is one greater.  

Consequently, we can solve Problem \ref{WaringPolynomials}.

\begin{corollary}
The smallest $s$ such that any generic polynomial in $n + 1$ variables of degree $d$ may be written as the sum of at most $s$ $d$\nth powers of linear forms is
\begin{equation*}
s = \begin{cases}
n + 1 &\text{ if }d=2\\
\left\lceil\frac{\binom{n+d}{d}}{n+1}\right\rceil+1 & \text{ if }(n,d)\in\{(2,4),(3,4),(4,3),(4,4)\}\\
\left\lceil\frac{\binom{n+d}{d}}{n+1}\right\rceil & \text{ otherwise.}
\end{cases}
\end{equation*}
\end{corollary}

In the next section, we will look at the secant variety of another variety to attempt to answer Waring's problem for completely decomposable forms.

\pagebreak

\section{Varieties of Completely Decomposable Forms}\label{CDFIntro}

In this section, we introduce varieties of completely decomposable forms and some of their basic properties.  We also discuss the connection between their secant varieties and Waring's problem for completely decomposable forms (Problem \ref{WaringCDF}).

\begin{definition}
A \define{variety of completely decomposable forms}, also known as a \define{Chow variety of zero cycles}, is a variety of the form
\begin{equation*}
\Split_d(\PP^n) = \{[f]\in\PP^{\binom{n+d}{d}-1}:f\text{ is a completely decomposable form}\}.
\end{equation*}
\end{definition}

Note that the variety of completely decomposable forms is in fact a variety.  Indeed, it is the image of the map $\varphi:(\PP^n)^d\rightarrow\PP^{\binom{n+d}{d}+1}$ defined by $([\ell_1],\ldots,[\ell_d])\mapsto[\ell_1\cdots\ell_d]$.  This is a polynomial (and therefore regular) map, as the coefficients of a product of linear forms are obtained by adding products of the coefficients of the linear forms.  For example, in the case of $\Split_2(\PP^1)$, the point $([a:b],[c:d])$ is sent by $\varphi$ to the point $[ac:ad+bc:bd]$.  Therefore, by \cite[Theorem 3.13]{Harris}, $\varphi((\PP^n)^d)$ is a variety.

The ideal of $\Split_d(\PP^n)$ is the kernel of the Foulkes-Howe map, and although minimal generators of this ideal are not yet known, it is defined set-theoretically by Brill's equations.\cite[Section 8.6]{Landsberg}

We summarize some important facts about varieties of completely decomposable forms.

\begin{lemma}
For all $n,d\in\NN$, $\dim\Split_d(\PP^n)=dn$.
\end{lemma}

\begin{proof}
Recall from the above the map $\varphi:(\PP^n)^d\rightarrow\PP^{\binom{n+d}{d}+1}$ defined by $([\ell_1],\ldots,[\ell_d])\mapsto[\ell_1\cdots\ell_d]$, which has the property that $\varphi((\PP^n)^d)=\Split_d(\PP^n)$.  Therefore, $\dim\Split_d(\PP^n)\allowbreak\leq\dim(\PP^n)^d=dn$.  However, for any $[f]=[\ell_1\cdots\ell_d]\in\Split_d(\PP^n)$, the fibre $\varphi^{-1}([f])$ consists of a finite set.  Indeed, it is the orbit of the set $\{[\ell_1],\ldots,[\ell_d]\}$ under the natural action of the symmetric group $\mathfrak S_d$ and thus has at most $d!$ elements.  Consequently, $\dim\Split_d(\PP^n)=dn$.
\end{proof}

\begin{lemma}\label{TangentToSplit}
Consider the point $[f]\in\Split_d(\PP^n)$, where $f=\ell_1\cdots\ell_d$ with $\ell_i$ a generic linear form for all $i$.  Then the affine cone of the tangent space to $\Split_d(\PP^n)$ at $[f]$ is
\begin{equation*}
\widehat T_{[f]}\Split_d(\PP^n)=\sum_{j=1}^d\ell_1\cdots\ell_{j-1}\ell_{j+1}\cdots\ell_dR_1,
\end{equation*}
where $R_1$ is the vector space of linear forms.
\end{lemma}

\begin{proof}
Consider the tangents to the curves of the following form, where $m_1,\ldots,m_d\in R_1$.
\begin{align*}
g(t)&=(\ell_1+m_1t)\cdots(\ell_d+m_dt)\\
&= f + t\sum_{j=1}^d\ell_1\cdots\ell_{j-1}m_j\ell_{j+1}\cdots\ell_d+O(t^2)\\
g'(t) &= \sum_{j=1}^d\ell_1\cdots\ell_{j-1}m_j\ell_{j+1}\cdots\ell_d+O(t)\\
g'(0) &= \sum_{j=1}^d\ell_1\cdots\ell_{j-1}m_j\ell_{j+1}\cdots\ell_d
\end{align*}

Note that these curves all lie on $\widehat\Split_d(\PP^n)$, and so the space
\begin{equation*}
V=\sum_{j=1}^d\ell_1\cdots\ell_{j-1}\ell_{j+1}\cdots\ell_dR_1
\end{equation*}
of all such tangent vectors is a subspace of $\widehat T_{[f]}\Split_d(\PP^n)$.

Note that the summands in this subspace intersect only in the line $\Span\{f\}$, and so we have
\begin{align*}
\dim V &= \dim\sum_{j=1}^d\ell_1\cdots\ell_{j-1}\ell_{j+1}\cdots\ell_dR_1\\
&= \sum_{j=1}^d\dim\ell_1\cdots\ell_{j-1}\ell_{j+1}\cdots\ell_dR_1-(d-1)\dim\Span\{f\}\\
&= d(n+1)-d+1\\
&= dn+1.
\end{align*}

Since $\dim\Split_d(\PP^n)=dn$, we have $\dim\widehat T_{[f]}\Split_d(\PP^n)=dn+1$, and thus it follows that $V=\widehat T_{[f]}\Split_d(\PP^n)$.
\end{proof}

Just as the Waring's problem for polynomials may be answered by looking at secant varieties of Veronese varieties, the Waring's problem for completely decomposable forms (Problem \ref{WaringCDF}) may be answered by looking at secant varieties of varieties of completely decomposable forms.  We ask, ``What is the smallest $s$ such that $\sigma_s(\Split_d(\PP^n))=\PP^{\binom{n+d}{d}-1}$?''

As in the Veronese case, this is trivial when $\sigma_s(\Split_d(\PP^n))$ is nondefective.  Since
\begin{equation*}
\dim\sigma_s(\Split_d(\PP^n)) = \min\left\{s(dn+1),\binom{n+d}{d}\right\}-1,
\end{equation*}
the secant variety will fill the ambient space when
\begin{align*}
s(dn+1)&\geq\binom{n+d}{d}\\
s&\geq\left\lceil\frac{\binom{n+d}{d}}{dn+1}\right\rceil.
\end{align*}

In order to fully solve Waring's problem for completely decomposable forms, we must identity all of the defective cases.  

There is one known defective family (see Proposition \ref{DEqualsTwo}).  In \cite{ArrondoBernardi}, Arrondo and Bernardi conjectured that this is the only one.

\begin{conjecture}\label{Conjecture}
The secant variety $\sigma_s(\Split_d(\PP^n))$ is nondefective except for the case when $d=2$ and $2\leq s\leq\frac{n}{2}$.
\end{conjecture}

The goal of this dissertation is to prove this conjecture for as many cases as possible.  In the next section, we summarize the previously known results.

\pagebreak

\section{Known Results}
\label{KnownResults}

We next summarize all of the cases of Conjecture \ref{Conjecture} which were previously known to be true.  First, note that every linear form is by definition completely decomposable, so $\Split_1(\PP^n)=\PP^n$, and consequently $\sigma_s(\Split_1(\PP^n))$ is nondefective for all $n,s\in\NN$.  Next, since $\kk$ is algebraically closed, every homogeneous polynomial over $\kk$ in two variables is completely decomposable, and so $\Split_d(\PP^1)=\PP^d$ and $\sigma_s(\Split_d(\PP^1))$ is nondefective for all $d,s\in\NN$.

In \cite{ArrondoBernardi}, Arrondo and Bernardi proved the following results.

\begin{proposition}\label{LargeN}
If $d\geq 3$ and $3(s-1)\leq n$, then $\sigma_s(\Split_d(\PP^n))$ is nondefective.  For all $n,s\in\NN$, $\dim\sigma_s(\mathbb G(1,n+1))=\dim\sigma_s(\Split_2(\PP^n))$.
\end{proposition}

The second statement shows that $\sigma_s(\Split_2(\PP^n))$ is defective if and only if $2\leq s\leq\frac{n}{2}$, as the dimensions of secant varieties of Grassmannians of lines are well known (see, for example, \cite{CGG}).  A direct proof of this case is shown in Proposition \ref{DEqualsTwo}.

In \cite{Shin}, Shin showed that the dimension of $\sigma_2(\Split_d(\PP^2))$ can be determined by the Hilbert function of the union of two linear star configurations of type $d$, leading to the following result.

\begin{theorem}
If $d\geq 3$, then $\sigma_2(\Split_d(\PP^2))$ is nondefective.
\end{theorem}

In \cite{Shin2}, Shin generalized this result to include $n\geq 3$, giving another proof for the $s=2$ case of Proposition \ref{LargeN}.

In \cite{Abo}, Abo used induction to complete the proof for 3-variable forms and provide partial proofs for 4-variable forms and cubic forms.

\begin{theorem}\label{AboResultNEqualsTwo}
For all $d,s\in\NN$, $\sigma_s(\Split_d(\PP^2))$ is nondefective.
\end{theorem}

\begin{theorem}\label{AboResult}
Consider the following functions.
\begin{align*}
s'_1(d) &=
\begin{cases}
\frac{1}{18}d^2+\frac{1}{6}d+1&\text{ if }d\equiv 0\pmod{6}\\
\frac{1}{18}d^2+\frac{2}{9}d-\frac{5}{18}&\text{ if }d\equiv 1\pmod{6}\\
\frac{1}{18}d^2+\frac{5}{18}d+\frac{2}{9}&\text{ if }d\equiv 2,5\pmod{6}\\
\frac{1}{18}d^2+\frac{1}{6}d&\text{ if }d\equiv 3\pmod{6}\\
\frac{1}{18}d^2+\frac{2}{9}d+\frac{2}{9}&\text{ if }d\equiv 4\pmod{6}\\
\end{cases}\\
s'_2(d) &=
\begin{cases}
\frac{1}{18}d^2+\frac{1}{3}d+1&\text{ if }d\equiv 0\pmod{6}\\
\frac{1}{18}d^2+\frac{7}{18}d+\frac{14}{9}&\text{ if }d\equiv 1\pmod{6}\\
\frac{1}{18}d^2+\frac{4}{9}d+\frac{8}{9}&\text{ if }d\equiv 2\pmod{6}\\
\frac{1}{18}d^2+\frac{1}{3}d+\frac{1}{2}&\text{ if }d\equiv 3\pmod{6}\\
\frac{1}{18}d^2+\frac{7}{18}d+\frac{5}{9}&\text{ if }d\equiv 4\pmod{6}\\
\frac{1}{18}d^2+\frac{4}{9}d+\frac{7}{18}&\text{ if }d\equiv 5\pmod{6}\\
\end{cases}
\end{align*}
If $n=3$ and $s\leq s'_1(d)$ or $s\geq s'_2(d)$, or $d=3$ and $s\leq s'_1(n)$ or $s\geq s'_2(n)$, then $\sigma_s(\Split_d(\PP^n))$ is nondefective.
\end{theorem}

We summarize the known results in Table \ref{NondefectiveSplit}.

\begin{table}[h]
\centering
\begin{tabular}{|c|c|c|}
\hline
$n$ & $d$ & $s$ \\
\hline
1,2& $\geq 1$ &$\geq 1$\\
$\geq 1$&1&$\geq 1$\\
$\geq 1$&2&$=1$ or $>\frac{n}{2}$\\
3&$\geq 1$&$\leq s'_1(d)$ or $\geq s'_2(d)$\\
$\geq 1$&3&$\leq s'_1(n)$ or $\geq s'_2(n)$\\
$\geq 3(s-1)$ & $\geq 3$ &$\geq 1$\\
\hline
\end{tabular}
\caption{Previously known nondefective cases of $\sigma_s(\Split_d(\PP^n))$}
\label{NondefectiveSplit}
\end{table}

\pagebreak

\chapter{Useful Tools}

In the following chapter, we introduce several techniques which will be vital in proving our results.

\section{Terracini's Lemma}
\label{TerraciniSection}

One of the earliest mathematicians interested in secant varieties of Veronese varieties was Alessandro Terracini, who in 1911 proved the following extremely useful result \cite{Terracini}.  It allows us to use linear algebra to answer questions regarding secant varieties.

\begin{lemma}[Terracini's lemma]\label{TerracinisLemma}
Let $X\subset\PP^N$ be a projective variety.  Choose generic points $p_1,\ldots,p_s\in X$ and a generic point $q\in\langle p_1,\ldots,p_s\rangle$.  Then
\begin{equation*}
\sigma_s(T_qX)=\langle T_{p_1}X,\ldots,T_{p_s}X\rangle.
\end{equation*}
\end{lemma}

\begin{corollary}\label{TerraciniCorollary}
Let $X\subset\PP^N$ be a projective variety.  If $\sigma_{s_1}(X)$ is nondefective and $s_1(\dim X+1)\leq N+1$, then $\sigma_s(X)$ is nondefective for all $s\leq s_1$.  If $\sigma_{s_2}(X)$ is nondefective and $s_2(\dim X+1)\geq N+1$, then $\sigma_s(X)$ is nondefective for all $s\geq s_2$.
\end{corollary}

\begin{proof}
Suppose there exists some $s\leq s_1$ such that $\sigma_s(X)$ is defective, \textit{i.e.}, $\dim\widehat\sigma_s(X)<s(\dim X+1)$.  Then, by Terracini's lemma,
\begin{align*}
\dim\sigma_{s_1}(X)&=\dim\sum_{i=1}^{s_1-1}\widehat T_{p_i}X-1\\
&\leq\dim\sum_{i=1}^s\widehat T_{p_i}X+\dim\sum_{i=s+1}^{s_1}\widehat T_{p_i}X-1\\
&<s(\dim X+1)+(s_1-s)(\dim X+1)-1\\
&=s_1(\dim X+1)-1.
\end{align*}
However, this contradicts the nondefectivity of $\sigma_{s_1}(X)$.

Note that $\sigma_1(X)\subset\sigma_2(X)\subset\cdots\sigma_{s_2}(X)=\PP^N$.  Therefore, if $s\geq s_2$, then $\sigma_s(X)=\PP^N$ also.
\end{proof}

For any $k\in\NN$ and $j\in\{1,\ldots,k\}$, define a map $\pi_j:R_1^{k}\mapsto\widehat{\Split_{k-1}}(\PP^n)$ by $(f_1,\cdots,f_k)\mapsto f_1\cdots f_{j-1}f_{j+1}\cdots f_k$.  Then, combining Terracini's lemma with Lemma~\ref{TangentToSplit}, we see that solving Waring's problem for completely decomposable forms may be reduced to a linear algebra problem.

\begin{proposition}\label{DimSecantVariety}
For any generic $\tuple f_1,\ldots,\tuple f_s\in R_1^{d}$,
\begin{equation*}
\dim\sigma_s(\Split_d(\PP^n))=\dim\sum_{i=1}^s\sum_{j=1}^d\pi_j(\tuple f_i)R_1-1.
\end{equation*}
\end{proposition}

A straightforward application of this fact is a direct proof of the $d=2$ case of Conjecture \ref{Conjecture}.

\begin{proposition}\label{DEqualsTwo}
The secant variety $\sigma_s(\Split_2(\PP^n))$ is defective if and only if $2\leq s\leq\frac{n}{2}$.  In particular, 
\begin{equation*}
\delta(\sigma_s(\Split_2(\PP^n)))=\begin{cases}
2s(s-1) &\text{ if }s\leq\frac{\binom{n+2}{2}}{2n+1}\\
\binom{n-2s+2}{2} & \text { if } \frac{\binom{n+2}{2}}{2n+1}\leq s\leq\frac{n}{2}\\
0 &\text{ otherwise.}\\
\end{cases}
\end{equation*}
\end{proposition}

\begin{proof}
Choose generic $\ell_0,\ldots,\ell_{2s-1}\in R_1$ and let $V=\sum_{i=0}^{2s-1}\ell_iR_1$.  Then, by Proposition \ref{DimSecantVariety},
\begin{equation*}
\dim\sigma_s(\Split_2(\PP^n))=\dim V-1.
\end{equation*}

If $2s>n$, then $V=R_2$, and so we achieve the expected dimension of $\binom{n+2}{2}-1$.

Suppose, on the other hand, that $2s\leq n$.  Then we can find linear forms $\ell_{2s},\ldots,\ell_n$ such that $\{\ell_0,\ldots,\ell_n\}$ forms a basis for $R_1$.  Note that $V$ is spanned by all the quadratic forms $\ell_i\ell_j$ except for those in which both $i$ and $j$ exceed $2s-1$.  Therefore,
\begin{align*}
\dim\sigma_s(\Split_2(\PP^n)) &= \binom{n+2}{2}-\binom{n-2s+2}{2}-1\\
&= s(2n+1)-2s(s-1)-1.\qedhere
\end{align*}
\end{proof}

\pagebreak

\section{Double Points and Hilbert Functions}

Consider a point $[f]\in\Split_d(\PP^n)$, and let $I_{[f]}$ be the ideal of all polynomials in $R$ which vanish at $[f]$ and let $I_{\Split_d(\PP^n)}$ be the ideal of all polynomials in $R$ which vanish on $\Split_d(\PP^n)$.  Then $I_{[f]}^2+I_{\Split_d(\PP^n)}$ defines the \define{double} or \define{2-fat point scheme} $2[f]$ in $\Split_d(\PP^n)$.  Double points are also known as \define{first infinitesimal neighborhoods}, as they contain infinitesimal tangent vectors.  (See, for example, \cite[Example II.9]{EisenbudHarris}.)  Consequently, we have the following lemma.
\begin{lemma}\label{DoublePoints}
Let $[f]\in\Split_d(\PP^n)$ be generic.  A hyperplane of $\PP^{\binom{n+d}{d}-1}$ contains $2[f]$ if and only if it contains $T_{[f]}\Split_d(\PP^n)$.
\end{lemma}

\begin{definition}\label{HilbertFunctionDefinition}
The \define{Hilbert function} of a scheme $Z\subset\PP^N$ with ideal sheaf $\sheaf I_Z$ is the function $h_{\PP^N}(Z,\cdot):\ZZ\rightarrow\ZZ$ defined by
\begin{equation*}
d\mapsto\binom{N+d}{d}-h^0(\PP^N,\sheaf I_Z(d)).
\end{equation*}
\end{definition}

If a hyperplane vanishes on $Z$ if and only if it vanishes on some linear subspace $\PP V$, then $h^0(\PP^N,\sheaf I_Z(1))=\dim V^\perp$, and so $h_{\PP^N}(Z,1)=N+1-\dim V^\perp=\dim V$.

Combining Lemma \ref{DoublePoints} with the discussion from Section \ref{CDFIntro} and Terracini's lemma, we have the following result.
\begin{proposition}\label{DoublePointTerracini}
Suppose $[f_1],\ldots,[f_s]\in\Split_d(\PP^n)$ are generic and $Z=\{2[f_1],\ldots,\allowbreak 2[f_s]\}$.  Then
\begin{equation*}
h_{\PP R_d}(Z,1)\leq \min\left\{s(dn+1),\binom{n+d}{d}\right\}.
\end{equation*}
In particular, we have equality if and only if $\sigma_s(\Split_d(\PP^n))$ is nondefective.
\end{proposition}

\begin{lemma}\label{HilbertFunctionInequality}
Suppose $Z\subset\PP^N$ is a scheme and $P=\PP^k$ is a linear subspace of $\PP^N$.  Then $h_{\PP^N}(Z,1)\geq h_{\PP^N}(Z\cup P,1)+h_P(Z\cap P,1)-k-1$.
\end{lemma}

\begin{proof}
We have the \define{Castelnuovo exact sequence}
\begin{equation*}
0\rightarrow \sheaf I_{Z\cup P}(1)\rightarrow \sheaf I_Z(1)\rightarrow \sheaf I_{Z\cap P,P}(1)\rightarrow 0.
\end{equation*}
This in turn gives a long exact sequence of cohomology groups
\begin{equation*}
0\rightarrow H^0(\PP^N,\sheaf I_{Z\cup P}(1))\rightarrow H^0(\PP^N,\sheaf I_Z(1))\rightarrow H^0(P,\sheaf I_{Z\cap P}(1)),
\end{equation*}
which gives us the inequalities
\begin{align*}
h^0(\PP^N,\sheaf I_Z(1))&\leq h^0(\PP^N,\sheaf I_{Z\cup P}(1))+h^0(P,\sheaf I_{Z\cap P}(1))\\
N+1-h_{\PP^N}(Z,1)&\leq N+1-h_{\PP^N}(Z\cup P,1)+k+1-h_P(Z\cap P,1)\\
h_{\PP^N}(Z,1)&\geq h_{\PP^N}(Z\cup P,1)+h_P(Z\cap P,1)-k-1.\qedhere
\end{align*}
\end{proof}

\pagebreak

\section{Computational Techniques}
\label{ComputationalTechniques}

By Lemma \ref{DimSecantVariety}, we see that the dimension of a secant variety of a completely decomposable form may be determined by calculating the dimension of a vector space, or equivalently, by calculating the rank of a matrix.

Since the matrices we will be using in this dissertation will be quite large, we use the computer algebra system Macaulay2 \cite{M2} to calculate their ranks.  Note that, rather than calculating the rank of a matrix directly, we will be finding the Gr\"{o}bner basis of an ideal generated by linear forms.  However, since Buchberger's algorithm is a generalization of Gaussian elimination, this amounts to the same thing.  (See, for example, \cite[Section 9.6]{DummitFoote}, for a discussion of Gr\"{o}bner bases and Buchberger's algorithm.)

To improve processing time, we will do our computations over a finite field rather than over an algebraically closed one.  This is possible due to the following lemma.

\begin{lemma}\label{WorkOverFiniteFields}
Consider a matrix $A\in M_{m\times n}(\ZZ)$, and let $\pi_p:M_{m\times n}(\ZZ)\rightarrow M_{m\times n}(\ZZ/(p))$ be the natural projection for some prime $p$.  If $\rank\pi_p(A)=r$, then $\rank A\geq r$.
\end{lemma}

\begin{proof}
Since $\rank\pi_p(A)=r$, there exists a nonzero $r\times r$ minor of $\pi_p(A)$, and consequently the corresponding $r\times r$ minor of $A$ must be nonzero.
\end{proof}

For the code to all of the Macaulay2 computations used in this dissertation, see Appendix \ref{M2Computations}.

\pagebreak

\chapter{Methods of Induction}

Recall that the problem of calculating the dimension of a secant variety may be reduced to calculating the dimension of a vector space by using Terracini's lemma.  

In the following chapter, we will introduce two separate methods of induction which will allow us to find the dimension of such a vector space by finding the dimensions of some easier base cases.

\section{Definitions}
\label{Definitions}

Recall from Proposition \ref{DimSecantVariety} that the vector space $\sum_{i=1}^s\sum_{j=1}^d\pi_j(\tuple f_i)R_1$, where $\tuple f_i\in R_1^d$, may be used to calculate the dimension of a secant variety of a variety of completely decomposable forms.  The two methods of induction in this chapter will involve adding this vector space to some other spaces.  In this section, we introduce notation to describe all of these spaces.

We will use the \define{backward difference operator} $\nabla$ defined by
\begin{equation*}
(\nabla^i_\ell f)(x) = \sum_{j=0}^i(-1)^j\binom{i}{j}f(x-j\ell)
\end{equation*}
for any function $f$.

Note that $\nabla^0_\ell f=f$ and, for all $i\geq 0$,
\begin{align}\label{BackwardDifferenceIdentity}
\begin{split}
(\nabla^{i+1}_\ell f)(d)+(\nabla^i_\ell f)(d-\ell) &= (\nabla^i_\ell f)(d)-(\nabla^i_\ell f)(d-\ell)+(\nabla^i_\ell f)(d-\ell)\\
&=(\nabla^i_\ell f)(d).\\
\end{split}
\end{align}

For any two tuples $\tuple f=(f_1,\ldots,f_{k_1})\in R_1^{k_1}$ and $\tuple g=(g_1,\ldots,g_{k_2})\in R_1^{k_2}$, we denote by $\tuple f|\tuple g$ the tuple $(f_1,\ldots,f_{k_1},g_1,\ldots,g_{k_2})\in R_1^{k_1+k_2}$.

Consider $n,d,\ell\in\NN$ and functions $s,t,u,v:\NN\rightarrow\ZZ_{\geq 0}$.  Fix an $i\in\{0,\ldots,n\}$.  Choose the following tuples of generic linear forms.
\begin{itemize}
\item For each $j\in\{1,\ldots,i\}$, let $\tuple g_j\in R_1^{\ell}$.
\item For each $j\in\{1,\ldots,i\}$ and $k\in\{1,\ldots,(\nabla^{i-1}_\ell s)(d-\ell)\}$, let $\tuple f_{j,k}\in R_1^{d-\ell}$.
\item For each $j\in\{1,\ldots,i\}$ and $k\in\{1,\ldots,(\nabla^{i-1}_\ell t)(d-\ell)\}$, let $\tuple f'_{j,k}\in R_1^{d-\ell+1}$.
\item For each $j\in\{1,\ldots,i\}$ and $k\in\{1,\ldots,(\nabla^{i-1}_\ell u)(d-\ell)\}$, let $\tuple f''_{j,k}\in R_1^{d-\ell+1}$.
\item For each $j\in\{1,\ldots,i\}$ and $k\in\{1,\ldots,(\nabla^{i-1}_\ell v)(d-\ell)\}$, let $\tuple f'''_{j,k}\in R_1^{d-\ell}$.
\item For each $j\in\{1,\ldots,(\nabla^i_\ell s )(d)\}$ let $\tuple f_j\in R_1^{d}$.
\item For each $j\in\{1,\ldots,(\nabla^i_\ell t)(d)\}$ let $\tuple f'_j\in R_1^{d+1}$.
\item For each $j\in\{1,\ldots,(\nabla^i_\ell u)(d)\}$ let $\tuple f''_j\in R_1^{d+1}$.
\item For each $j\in\{1,\ldots,(\nabla^i_\ell v)(d)\}$ let $\tuple f'''_j\in R_1^{d}$.
\end{itemize}

We then define the following subspace of $R_d$.
\begin{align*}
&A_i(n,d,\ell,s,t,u,v) =\sum_{j=1}^i\left(\prod \tuple g_j\right)R_{d-\ell}\\
&\quad+\sum_{j=1}^i\sum_{k=1}^{(\nabla^{i-1}_\ell s)(d-\ell)}\sum_{m=1}^d\pi_m(\tuple f_{j,k}|\tuple g_j)R_1+\sum_{j=1}^i\sum_{k=1}^{(\nabla^{i-1}_\ell t)(d-\ell)}\Span\{\pi_1(\tuple f'_{j,k}|\tuple g_j)\}\\
&\quad+\sum_{j=1}^i\sum_{k=1}^{(\nabla^{i-1}_\ell u)(d-\ell)}\sum_{m=1}^{d+1}\Span\{\pi_m(\tuple f''_{j,k}|\tuple g_j)\}+\sum_{j=1}^i\sum_{k=1}^{(\nabla^{i-1}_\ell v)(d-\ell)}\pi_1(\tuple f'''_{j,k}|\tuple g_j)R_1\\
&\quad+\sum_{j=1}^{(\nabla^i_\ell s )(d)}\sum_{m=1}^d\pi_m(\tuple f_j)R_1+\sum_{j=1}^{(\nabla^i_\ell t )(d)}\Span\{\pi_1(\tuple f'_j)\}\\
&\quad+\sum_{j=1}^{(\nabla^i_\ell u )(d)}\sum_{m=1}^{d+1}\Span\{\pi_m(\tuple f''_j)\}+\sum_{j=1}^{(\nabla^i_\ell v )(d)}\pi_1(\tuple f'''_j)R_1.
\end{align*}

Note that if $i=0$, then the first five summands are zero.  In particular, we have
\begin{align*}
A_0(n,d,\ell,s,t,u,v)&=\sum_{j=1}^{s(d)}\sum_{m=1}^d\pi_m(\tuple f_j)R_1+\sum_{j=1}^{t(d)}\Span\{\pi_1(\tuple f'_j)\}\\
&\quad+\sum_{j=1}^{u(d)}\sum_{m=1}^{d+1}\Span\{\pi_m(\tuple f''_j)\}+\sum_{j=1}^{v(d)}\pi_1(\tuple f'''_j)R_1.
\end{align*}

For the sake of brevity, we will denote this space by $A_i(d)$ if $n,\ell,s,t,u,$ and $v$ are understood (as in Section \ref{RestrictionInduction}).  Also, if $i=0$, $\ell$ is unnecessary, so we may denote this space by $A(n,d,s,t,u,v)$ (as in Section \ref{SplittingInduction}).

Note that $A(n,d,s,0,0,0)$ is precisely the vector space from Proposition \ref{DimSecantVariety}, \textit{i.e.}, 
\begin{equation*}
\dim\sigma_{s(d)}(\Split_d(\PP^n))=\dim A(n,d,s,0,0,0)-1.
\end{equation*}

Next, we define a function
\begin{align*}
a_i(n,d,\ell,s,t,u,v) &= \sum_{j=1}^i(-1)^{j-1}\binom{i}{j}\binom{n+d-j\ell}{d-j\ell}+i\ell n(\nabla^{i-1}_\ell s)(d-\ell)\\
&\quad + i\ell(\nabla^{i-1}_\ell u)(d-\ell)+(dn+1)(\nabla^i_\ell s)(d)+(\nabla^i_\ell t)(d)\\
&\quad + (d+1)(\nabla^i_\ell u)(d) + (n+1)(\nabla^i_\ell v)(d).
\end{align*}

As above, we may abbreviate this as $a_i(d)$ or $a(n,d,s,t,u,v)$ if desired.

\begin{lemma}\label{AiExpdim}
For any $i,n,d,\ell,s,t,u,v$, we have 
\begin{equation*}
\dim A_i(d) \leq \min\left\{a_i(d),\binom{n+d}{d}\right\}.
\end{equation*}
\end{lemma}

\begin{proof}
Certainly, since $A_i(d)$ is a subspace of $R_d$, $\dim A_i(d)\leq\binom{n+d}{d}$.

Note that, for any subset $I\subset\{1,\ldots,i\}$, $\bigcap_{j\in I}\left(\prod \tuple g_j\right)R_{d-\ell}=\left(\prod_{j\in I}\left(\prod \tuple g_j\right)\right)R_{d-|I|\ell}$, which has dimension $\binom{n+d-|I|\ell}{d-|I|\ell}$.  Therefore, since the $\tuple g_j$ are chosen to be generic, by the inclusion-exclusion principle, we have
\begin{equation*}
\dim \sum_{j=1}^i\left(\prod \tuple g_j\right)R_{d-\ell} =\sum_{j=1}^i(-1)^{j-1}\binom{i}{j}\binom{n+d-j\ell}{d-j\ell}.
\end{equation*}

Next note that, for any $j\in\{1,\ldots,i\}$ and $k\in\{1,\ldots,(\nabla^{i-1}_\ell s)(d-\ell)\}$, we have by Lemma \ref{TangentToSplit}
\begin{align*}
\left(\prod \tuple g_j\right)R_{d-\ell}\cap\sum_{m=1}^d\pi_m(\tuple f_{j,k}|\tuple g_j)R_1&=\sum_{m=1}^{d-\ell}\pi_m(\tuple f_{j,k}|\tuple g_j)R_1\\
&=\left(\prod \tuple g_j\right)\widehat T_{\left[\prod \tuple f_{j,k}|\tuple g_j\right]}\Split_{d-\ell}(\PP^n),\\
\intertext{so}
\dim\left(\left(\prod \tuple g_j\right)R_{d-\ell}\cap\sum_{m=1}^d\pi_m(\tuple f_{j,k}|\tuple g_j)R_1\right)&=(d-\ell)n+1.
\end{align*}

Now, for any $j\in\{1,\ldots,i\}$ and $k\in\{1,\ldots,(\nabla^{i-1}_\ell u)(d-\ell)\}$, we have
\begin{equation*}
\left(\prod \tuple g_j\right)R_{d-\ell}\cap\sum_{m=1}^{d+1}\Span\{\pi_m(\tuple f''_{j,k}|\tuple g_j)\}=\sum_{m=1}^{d-\ell+1}\Span\{\pi_m(\tuple f''_{j,k}|\tuple g_j)\},
\end{equation*}
and so
\begin{equation*}
\dim\left(\left(\prod \tuple g_j\right)R_{d-\ell}\cap\sum_{m=1}^{d+1}\Span\{\pi_m(\tuple f''_{j,k}|\tuple g_j)\}\right)=d-\ell+1.
\end{equation*}

Note that, for each $j\in\{1,\ldots,i\}$ and $k\in\{1,\ldots,(\nabla^{i-1}_\ell t)(d-\ell)\}$,
\begin{equation*}
\Span\{\pi_1(\tuple f'_{j,k}|\tuple g_j)\}\subset\left(\prod \tuple g_j\right)R_{d-\ell}.
\end{equation*}

Note also that, for each $j\in\{1,\ldots,i\}$ and $k\in\{1,\ldots,(\nabla^{i-1}_\ell v)(d-\ell)\}$,
\begin{equation*}
\pi_1(\tuple f'''_{j,k}|\tuple g_j)R_1\subset\left(\prod \tuple g_j\right)R_{d-\ell}.
\end{equation*}

Therefore,
\begin{align*}
&\dim A_i(d)\leq\dim\sum_{j=1}^i\left(\prod \tuple g_j\right)R_{d-\ell}\\
&\quad+\dim\sum_{j=1}^i\sum_{k=1}^{(\nabla^{i-1}_\ell s)(d-\ell)}\sum_{m=1}^d\pi_m(\tuple f_{j,k}|\tuple g_j)R_1+\dim\sum_{j=1}^i\sum_{k=1}^{(\nabla^{i-1}_\ell u)(d-\ell)}\sum_{m=1}^{d+1}\Span\{\pi_m(\tuple f''_{j,k}|\tuple g_j)\}\\
&\quad-\sum_{j=1}^i\sum_{k=1}^{(\nabla^{i-1}_\ell s)(d-\ell)}\dim\left(\left(\prod \tuple g_j\right)R_{d-\ell}\cap\sum_{m=1}^d\pi_m(\tuple f_{j,k}|\tuple g_j)R_1\right)\\
&\quad-\sum_{j=1}^i\sum_{k=1}^{(\nabla^{i-1}_\ell u)(d-\ell)}\dim\left(\left(\prod \tuple g_j\right)R_{d-\ell}\cap\sum_{m=1}^{d+1}\Span\{\pi_m(\tuple f''_{j,k}|\tuple g_j)\}\right)\\
&\quad+\dim\sum_{j=1}^{(\nabla^i_\ell s )(d)}\sum_{m=1}^d\pi_m(\tuple f_j)R_1+\dim\sum_{j=1}^{(\nabla^i_\ell t )(d)}\Span\{\pi_1(\tuple f'_j)\}\\
&\quad+\dim\sum_{j=1}^{(\nabla^i_\ell u )(d)}\sum_{m=1}^{d+1}\Span\{\pi_m(\tuple f''_j)\}+\dim\sum_{j=1}^{(\nabla^i_\ell v )(d)}\pi_1(\tuple f'''_j)R_1\\
&\leq\sum_{j=1}^i(-1)^{j-1}\binom{i}{j}\binom{n+d-j\ell}{d-j\ell}+i(dn+1)(\nabla^{i-1}_\ell s)(d-\ell)+i(d+1)(\nabla^{i-1}_\ell u)(d-\ell)\\
&\quad-i((d-\ell)n+1)(\nabla^{i-1}_\ell s)(d-\ell)-i(d-\ell+1)(\nabla^{i-1}_\ell u)(d-\ell)\\
&\quad+(dn+1)(\nabla^i_\ell s)(d)+(\nabla^i_\ell t)(d)+(d+1)(\nabla^i_\ell u)(d)+(n+1)(\nabla^i_\ell v)(d)\\
&=a_i(d).\qedhere
\end{align*}
\end{proof}

\begin{definition}
We define the statement $\statement A_i(n,d,\ell,s,t,u,v)$ to be \define{true} if
\begin{equation*}
\dim A_i(d)=\min\left\{a_i(d),\binom{n+d}{d}\right\}
\end{equation*}
and \define{false} otherwise.  We may abbreviate this as $\statement A_i(d)$ or $\statement A(n,d,s,t,u,v)$ (if $i=0$) if desired.

In addition, we define the \define{abundancy} of $\statement A_i(d)$ as follows.
\begin{itemize}
\item If $a_i(d)\leq\binom{n+d}{d}$, then $\statement A_i(d)$ is \define{subabundant}.
\item If $a_i(d)\geq\binom{n+d}{d}$, then $\statement A_i(d)$ is \define{superabundant}.
\item If $a_i(d)=\binom{n+d}{d}$, then $\statement A_i(d)$ is \define{equiabundant}.
\end{itemize}
\end{definition}

Note that, by Proposition \ref{DimSecantVariety}, $\statement A(n,d,s,0,0,0)$ is true if and only if $\sigma_{s(d)}(\Split_d(\PP^n))$ is nondefective.

\pagebreak

\section{Splitting Induction}
\label{SplittingInduction}

The following method of induction was inspired by one used by Abo, Ottaviani, and Peterson in \cite{AOP} to study secant varieties of Segre varieties.

Consider the subspace $U=\Span\{x_1,\ldots,x_n\}$ of $R_1$.  Note that $R_d=x_0R_{d-1}\oplus S_dU$.  Our goal is to split $A(n,d,s,0,0,0)$ into a direct sum of smaller vector spaces using this fact.  However, we cannot use the natural map $R_d\rightarrow R_{d-1}\oplus S_dU$, as in general, the images of completely decomposable forms under this map will not themselves be completely decomposable.  We therefore restrict our attention to a smaller subset of completely decomposable forms.  Consequently, the following induction method will only work in the subabundant case.

First, we motivate the idea with an example.  Suppose we want to find the dimension of $\sigma_2(\Split_3(\PP^5))$.  By Proposition \ref{DimSecantVariety}, we need only find the dimension of $\widehat T_{[x_0x_1x_2]}\Split_3(\PP^5)+\widehat T_{[x_3x_4x_5]}\Split_3(\PP^5)$.  We have
\begin{align*}
&\widehat T_{[x_0x_1x_2]}\Split_3(\PP^5)+\widehat T_{[x_3x_4x_5]}\Split_3(\PP^5)\\
&\quad=x_0x_1R_1+x_0x_2R_1+x_1x_2R_1\\
&\quad\quad+x_3x_4R_1+x_3x_5R_1+x_4x_5R_1\\
&\quad=(x_1x_2U+x_3x_4U+x_3x_5U+x_4x_5U)\\
&\quad\quad\oplus x_0(x_1U+x_2U+\Span\{x_3x_4,x_3x_5,x_4x_5\})\\
&\quad\quad\oplus x_0^2\Span\{x_1,x_2\}\\
&\quad=(x_1x_2U+\widehat T_{[x_3x_4x_5]}\Split_3(\PP^4))\\
&\quad\quad\oplus(\widehat T_{[x_1x_2]}\Split_2(\PP^4)+\Span\{x_3x_4,x_3x_5,x_4x_5\})\\
&\quad\quad\oplus x_0^2\Span\{x_1,x_2\}.
\end{align*}

Note that, in addition to two tangent spaces, some other vector spaces appear in this splitting.  These vector spaces account for some of the extra terms which appear in the definition of $A(n,d,s,t,u,v)$.

We now generalize this idea.

For any $k\in\NN$ and $j\in\{1,\ldots,k\}$, define a map $\rho_j:R_1^{k}\mapsto R_1^{k-1}$ by $(f_1,\cdots,f_k)\mapsto(f_1,\ldots f_{j-1},f_{j+1},\ldots,f_k)$.

\begin{theorem}[Splitting induction]\label{SplittingInductionTheorem}
Suppose $n\geq 2$, $d\geq 3$, $s=s'+s''$, $t=t'+t''$, $u=u'+u''$, and $v=v'+v''$.  If $\statement A(n-1,d,s'',t''+u',u'',s'+v'')$, $\statement A(n-1,d-1,s',t'+v'',s''+u',v')$, and $\statement A(n-1,d-2,0,v',s',0)$ are all true and subabundant, then $\statement A(n,d,s,t,u,v)$ is true and subabundant.  In particular, if $t=u=v=0$, then $\sigma_s(\Split_d(\PP^n))$ is nondefective.
\end{theorem}

\begin{proof}
We specialize the vector space $A(n,d,s,t,u,v)$ as follows.
\begin{itemize}
\item Choose generic $\tuple f_j\in\Span\{x_0\}\times U^{d-1}$ if $j\in\{1,\ldots,s'(d)\}$ and $\tuple f_j\in U^{d}$ if $j\in\{s'(d)+1,\ldots,s(d)\}$.
\item Choose generic $\tuple f'_j\in R_1\times\Span\{x_0\}\times U^{d-1}$ if $j\in\{1,\ldots,t'(d)\}$ and $\tuple f'_j\in U^{d}$ if $j\in\{t'(d)+1,\ldots,t(d)\}$.
\item Choose generic $\tuple f''_j\in\Span\{x_0\}\times U^{d}$ if $j\in\{1,\ldots,u'(d)\}$ and $\tuple f''_j\in U^{d}$ if $j\in\{u'(d)+1,\ldots,u(d)\}$.
\item Choose generic $\tuple f'''_j\in R_1\times\Span\{x_0\}\times U^{d-2}$ if $j\in\{1,\ldots,v'(d)\}$ and $\tuple f'''_j\in U^{d}$ if $j\in\{v'(d)+1,\ldots,v(d)\}$.
\end{itemize}

Recall from the definition that $A(n,d,s,t,u,v) = A(n,d,s,0,0,0)+A(n,d,0,t,0,0)+A(n,d,0,0,u,0)+A(n,d,0,0,0,v)$.

Then
\begin{align*}
A(n,d,s,0,0,0) &= \sum_{j=1}^{s'(d)}\sum_{m=1}^d\pi_m(\tuple f_j)R_1\\
&= \sum_{j=1}^{s'(d)}\left(\pi_1(\tuple f_j)R_1+x_0\sum_{m=2}^d\pi_m(\rho_1(\tuple f_j))R_1\right)\\
&\quad+\sum_{j=s'(d)+1}^{s(d)}\left(\sum_{m=1}^d\pi_m(\tuple f_j)U+x_0\sum_{m=1}^d\Span\{\pi_m(\tuple f_j)\}\right)\\
&=\left(\sum_{j=1}^{s'(d)}\pi_1(\tuple f_j)U+\sum_{j=s'(d)+1}^{s(d)}\sum_{m=1}^d\pi_m(\tuple f_j)U\right)\\
&\quad\oplus x_0\left(\sum_{j=1}^{s'(d)}\sum_{m=2}^d\pi_m(\rho_1(\tuple f_j))U+\sum_{j=s'(d)+1}^{s(d)}\sum_{m=1}^d\Span\{\pi_m(\tuple f_j)\}\right)\\
&\quad\oplus x_0^2\left(\sum_{j=1}^{s'(d)}\sum_{m=2}^d\Span\{\pi_m(\rho_1(\tuple f_j))\}\right)\\
&\cong A(n-1,d,s'',0,0,s')\oplus A(n-1,d-1,s',0,s'',0)\\
&\quad\oplus A(n-1,d-2,0,0,s',0),\\
A(n,d,0,t,0,0) &= \sum_{j=1}^{t(d)}\Span\{\pi_1(\tuple f'_j)\}\\
&=\sum_{j=t'(d)+1}^{t(d)}\Span\{\pi_1(\tuple f'_j)\}\oplus x_0\sum_{j=1}^{t'(d)}\Span\{\pi_1(\rho_1(\tuple f'_j))\}\\
&\cong A(n-1,d,0,t'',0,0)\oplus A(n-1,d-1,0,t',0,0),\\
A(n,d,0,0,u,0) &= \sum_{j=1}^{u(d)}\sum_{m=1}^{d+1}\Span\{\pi_m(\tuple f''_j)\}\\
&= \sum_{j=1}^{u'(d)}\left(\Span\{\pi_1(\tuple f''_j)\}+x_0\sum_{m=2}^{d+1}\Span\{\pi_m(\rho_1(\tuple f''_j))\}\right)\\
&\quad +\sum_{j=u'(d)+1}^{u(d)}\sum_{m=1}^{d+1}\Span\{\pi_m(\tuple f''_j)\}\\
&= \left(\sum_{j=1}^{u'(d)}\Span\{\pi_1(\tuple f''_j)\}+\sum_{j=u'(d)+1}^{u(d)}\sum_{m=1}^{d+1}\Span\{\pi_m(\tuple f''_j)\}\right)\\
&\quad\oplus x_0\sum_{j=1}^{u'(d)}\sum_{m=2}^{d+1}\Span\{\pi_m(\rho_1(\tuple f''_j))\}\\
&\cong A(n-1,d,0,u',u'',0)\oplus A(n-1,d-1,0,0,u',0)\text{, and}\\
A(n,d,0,0,0,v) &= \sum_{j=1}^{v(d)}\pi_m(\tuple f'''_j)R_1\\
&=x_0\sum_{j=1}^{v'(d)}\pi_1(\rho_2(\tuple f'''_j))R_1+\sum_{j=v'(d)+1}^{v(d)}\pi_1(\tuple f'''_j)U\\
&\quad+x_0\sum_{j=v'(d)+1}^{v(d)}\Span\{\pi_1(\tuple f'''_j)\}\\
&=\left(\sum_{j=v'(d)+1}^{v(d)}\pi_1(\tuple f'''_j)U\right)\\
&\quad\oplus x_0\left(\sum_{j=1}^{v'(d)}\pi_1(\rho_2(\tuple f'''_j))U+\sum_{j=v'(d)+1}^{v(d)}\Span\{\pi_1(\tuple f'''_j)\}\right)\\
&\quad\oplus x_0^2\left(\sum_{j=1}^{v'(d)}\Span\{\pi_1(\rho_2(\tuple f'''_j))\}\right)\\
&\cong A(n-1,d,0,0,0,v'')\oplus A(n-1,d-1,0,v'',0,v')\\
&\quad\oplus A(n-1,d-2,0,v',0,0).
\end{align*}

Adding these results together, we get
\begin{align*}
&A(n,d,s,t,u,v) =\\
&\quad A(n-1,d,s'',t''+u',u'',s'+v'')\oplus A(n-1,d-1,s',t'+v'',s''+u',v')\\
&\quad\oplus A(n-1,d-2,0,v',s',0).
\end{align*}

By assumption, we have
\begin{align*}
&\dim A(n,d,s,t,u,v) = \\
&\quad a_0(n-1,d,s'',t''+u',u'',s'+v'')+a_0(n-1,d-1,s',t'+v'',s''+u',v')\\
&\quad+a_0(n-1,d-2,0,v',s',0)\\
&=s''(d)((n-1)d+1)+t''(d)+u'(d)+u''(d)(d+1)+(s'(d)+v''(d))n\\
&\quad+s'(d)((n-1)(d-1)+1)+t'(d)+v''(d)+(s''(d)+u'(d))d+v'(d)n\\
&\quad+v'(d)+s'(d)(d-1)\\
&=s(d)(dn+1)+t(d)+u(d)(d+1)+v(d)(n+1)\\
&=a_0(n,d,s,t,u,v).
\end{align*}

Consequently, $\statement A(n,d,s,t,u,v)$ is true and subabundant.
\end{proof}

\pagebreak

\section{Restriction Induction with Fixed Dimension}
\label{RestrictionInduction}

In this section, we outline a method of induction which is a generalization of the one used by Abo to complete the proof of Conjecture \ref{Conjecture} for forms in three variables and provide a partial proof for forms in four variables \cite{Abo}.  This was in turn adapted from one initially developed by Terracini during his work on the proof of Theorem~\ref{AlexanderHirschowitzTheorem} \cite{Terracini2}.

We sketch the idea with an example before giving the complete proof.  Refer to Figure \ref{InductionExample} for a visual reference.  Suppose we want to prove the nondefectivity of $\sigma_{52}(\Split_{28}(\PP^3))$.  By Proposition \ref{DoublePointTerracini}, we can do this by considering a scheme $Z$ consisting of 52 double points on $\Split_{28}(\PP^3)\subset\PP^{\binom{3+28}{28}-1}=\PP^{4494}$.  We know that $h_{\PP^{4494}}(Z,1)\leq\min\left\{52(28\cdot 3+1),\binom{3+28}{28}\right\}=4420$ and we want to show that in fact we have equality.

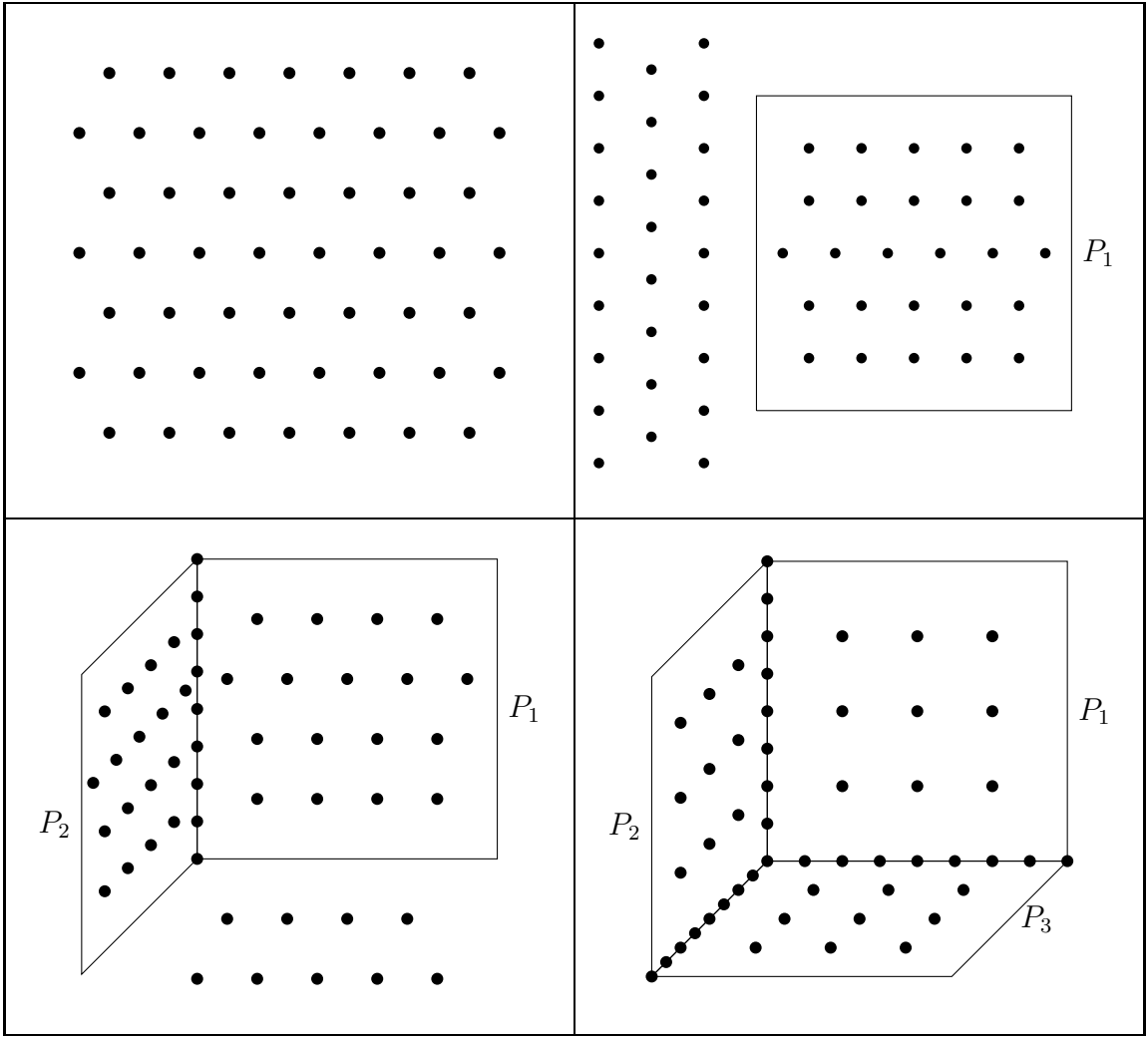
\begin{figure}[h] 
\begin{tabular}{|m{2.82in}|m{2.82in}|}
\hline
\begin{center}
\begin{tikzpicture}[scale=.8]
\foreach \x in {1,3,5,7,9,11,13} do \foreach \y in {0,2,4,6} do \fill (.5*\x,\y) circle (.1);
\foreach \x in {0,2,4,6,8,10,12,14} do \foreach \y in {1,3,5} do \fill (.5*\x,\y) circle (.1);
\end{tikzpicture}
\end{center}
&
\begin{center}
\begin{tikzpicture}[scale=.7]
\draw (0,0) -- (6,0) -- (6,6) -- (0,6) -- (0,0);
\draw (6,3) node[right] {$P_1$};
\foreach \x in {1,2,3,4,5} do \foreach \y in {1,2,4,5} do \fill (\x,\y) circle (.1);
\foreach \x in {.5,1.5,2.5,3.5,4.5,5.5} do \fill (\x,3) circle (.1);
\foreach \x in {-1,-3} do \foreach \y in {-1,...,7} do \fill (\x,\y) circle (.1);
\foreach \y in {-.5,.5,1.5,2.5,3.5,4.5,5.5,6.5} do \fill (-2,\y) circle (.1);
\end{tikzpicture}
\end{center}
\\\hline
\begin{center}
\begin{tikzpicture}[scale=.8]
\draw (0,0) -- (5,0) -- (5,5) -- (0,5) -- (0,0);
\foreach \x in {1,...,4} do \foreach \y in {1,2,4} do \fill (\x,\y) circle (.1);
\foreach \x in {.5,1.5,2.5,3.5,4.5} do \fill (\x,3) circle (.1);
\draw (5,2.5) node[right] {$P_1$};
\draw (0,0) -- (0,0,5) -- (0,5,5) -- (0,5,0) -- (0,0);
\draw (0,2.5,5) node[left] {$P_2$};
\foreach \x in {1,...,4} do \foreach \y in {1,2,4} do \fill (0,\y,\x) circle (.1);
\foreach \x in {.5,1.5,2.5,3.5,4.5} do \fill (0,3,\x) circle (.1);
\foreach \x in {0,...,8} do \fill (0,{(5/8)*\x}) circle (.1);
\foreach \x in {0,...,4} do \fill (\x,-2) circle (.1);
\foreach \x in {.5,1.5,2.5,3.5} do \fill (\x,-1) circle (.1);
\end{tikzpicture}
\end{center}
&
\begin{center}
\begin{tikzpicture}
\draw (0,0) -- (4,0) -- (4,4) -- (0,4) -- (0,0);
\fill (0,0) circle (.08);
\foreach \y in {.5,1,1.5,2,2.5,3,3.5,4} do \fill (0,\y) circle (.08);
\foreach \x in {1,2,3} do \foreach \y in {1,2,3} do \fill (\x,\y) circle (.08);
\draw (0,0) -- (0,0,4) -- (0,4,4) -- (0,4,0) -- (0,0);
\foreach \y in {.5,1,1.5,2,2.5,3,3.5,4} do \fill (0,0,\y) circle (.08);
\foreach \x in {1,2,3} do \foreach \y in {1,2,3} do \fill (0,\x,\y) circle (.08);
\draw (0,0) -- (0,0,4) -- (4,0,4) -- (4,0) -- (0,0);
\foreach \y in {.5,1,1.5,2,2.5,3,3.5,4} do \fill (\y,0) circle (.08);
\foreach \x in {1,2,3} do \foreach \y in {1,2,3} do \fill (\x,0,\y) circle (.08);
\draw (4,2) node[right] {$P_1$};
\draw (0,2,4) node[left] {$P_2$};
\draw (4,0,2) node[right] {$P_3$};
\end{tikzpicture}
\end{center}\\\hline
\end{tabular}
\caption{Example of restriction induction\label{InductionExample}}
\end{figure}

Let $f_1$, $f_2$, and $f_3$ be generic completely decomposable forms of degree 9.  For each $i\in\{1,2,3\}$, let $P_i=\{[g]\in\PP^{4494}:f_i\text{ divides }g\}$.  Note that $P_i\cong\PP^{\binom{3+19}{19}-1}=\PP^{1539}$ by the map $[g]\mapsto\left[\frac{g}{f_i}\right]$.  For this same reason, $\Split_{28}(\PP^3)\cap P_i\cong \Split_{19}(\PP^3)$.  Also, note that for $i\neq j$, $P_i\cap P_j\cong\PP^{\binom{3+10}{10}-1}=\PP^{285}$ by the map $[g]\mapsto\left[\frac{g}{f_if_j}\right]$ and $\Split_{28}(\PP^3)\cap P_i\cap P_j\cong\Split_{10}(\PP^3)$.  Finally $P_1\cap P_2\cap P_3\cong\PP^{\binom{3+1}{1}-1}=\PP^3$ by $[g]\mapsto\left[\frac{g}{f_1f_2f_3}\right]$ and $\Split_{28}(\PP^3)\cap P_1\cap P_2\cap P_3\cong\Split_1(\PP^3)$.

We increasingly specialize $Z$ by declaring that some of the double points lie in one of the above intersections.  First, specialize 26 of the double points to $\Split_{28}(\PP^3)\cap P_1=\Split_{19}(\PP^3)$.  By Lemma \ref{HilbertFunctionInequality}, we have 
\begin{align*}
h_{\PP^{4494}}(Z,1)&\geq h_{\PP^{4494}}(Z\cup P_1,1)+h_{P_1}(Z\cap P_1,1)-\binom{3+28-9}{28-9}\\
&=h_{\PP^{4494}}(Z\cup P_1,1)+h_{P_1}(Z\cap P_1,1)-1540.
\end{align*}

Next, specialize another 17 double points to $\Split_{28}(\PP^3)\cap P_2=\Split_{19}(\PP^3)$, and choose 9 of the double points which have already been specialized to $\Split_{28}(\PP^3)\cap P_1$ to specialize to $\Split_{28}(\PP^3)\cap P_1\cap P_2=\Split_{10}(\PP^3)$.
\begin{align*}
h_{\PP^{4494}}(Z\cup P_1,1) &\geq h_{\PP^{4494}}(Z\cup P_1\cup P_2,1)+h_{P_2}((Z\cup P_1)\cap P_2,1)-1540\\
h_{P_1}(Z\cap P_1,1) &\geq h_{P_1}((Z\cup P_2)\cap P_1,1)+h_{P_1\cap P_2}(Z\cap P_1\cap P_2,1)-\binom{3+28-18}{28-18}\\
&= h_{P_1}((Z\cup P_2)\cap P_1,1)+h_{P_1\cap P_2}(Z\cap P_1\cap P_2,1)-286
\end{align*}
and so
\begin{align*}
h_{\PP^{4494}}(Z,1)&\geq h_{\PP^{4494}}(Z\cup P_1\cup P_2,1)+2h_{P_2}((Z\cup P_1)\cap P_2,1)\\
&\quad+h_{P_1\cap P_2}(Z\cap P_1\cap P_2,1)-3366.
\end{align*}

Finally, we specialize the remaining 9 double points to $\Split_{28}(\PP^3)\cap P_3=\Split_{19}(\PP^3)$.  We further specialize 8 of the double points on $\Split_{28}(\PP^3)\cap P_1$ to $\Split_{28}(\PP^3)\cap P_1\cap P_3=\Split_{10}(\PP^3)$,  8 of the double points on $\Split_{28}(\PP^3)\cap P_2$ to $\Split_{28}(\PP^3)\cap P_2\cap P_3=\Split_{10}(\PP^3)$, and one double point on $\Split_{28}(\PP^3)\cap P_1\cap P_2$ to $\Split_{28}(\PP^3)\cap P_1\cap P_2\cap P_3=\Split_1(\PP^3)$.  Then,
\begin{align*}
h_{\PP^{4494}}(Z\cup P_1\cup P_2,1) &\geq h_{\PP^{4494}}(Z\cup P_1\cup P_2\cup P_3,1)\\
&\quad+h_{P_3}((Z\cup P_1\cup P_2)\cap P_3,1)-1540,\\
h_{P_2}((Z\cup P_1)\cap P_2,1) &\geq h_{P_2}((Z\cup P_1\cup P_3)\cap P_2,1)\\
&\quad+h_{P_2\cap P_3}((Z\cup P_1)\cap P_2\cap P_3,1)-286,\text{ and}\\
h_{P_1\cap P_2}(Z\cap P_1\cap P_2,1) &\geq h_{P_1\cap P_2}((Z\cup P_3)\cap P_1\cap P_2,1)\\
&+h_{P_1\cap P_2\cap P_3}(Z\cap P_1\cap P_2\cap P_3,1)-\binom{3+28-27}{28-27}\\
&=h_{P_1\cap P_2}((Z\cup P_3)\cap P_1\cap P_2,1)\\
&+h_{P_1\cap P_2\cap P_3}(Z\cap P_1\cap P_2\cap P_3,1)-4,
\end{align*}
and so
\begin{align*}
h_{\PP^{4494}}(Z,1) & \geq h_{\PP^{4494}}(Z\cup P_1\cup P_2\cup P_3,1)+3h_{P_3}((Z\cup P_1\cup P_2)\cap P_3,1)&\\
&\quad+3h_{P_2\cap P_3}((Z\cup P_1)\cap P_2\cap P_3,1)\\
&+h_{P_1\cap P_2\cap P_3}(Z\cap P_1\cap P_2\cap P_3,1)-5482.
\end{align*}

At this point, we calculate the above Hilbert functions using Macaulay2, yielding
\begin{align*}
h_{\PP^{4494}}(Z\cup P_1\cup P_2\cup P_3,1)&=4495\\
h_{P_3}((Z\cup P_1\cup P_2)\cap P_3,1)&=1522\\
h_{P_2\cap P_3}((Z\cup P_1)\cap P_2\cap P_3,1)&=279\\
h_{P_1\cap P_2\cap P_3}(Z\cap P_1\cap P_2\cap P_3,1)&=4.
\end{align*}

Putting this all together, we get $h_{\PP^{4494}}(Z,1)\geq 4420$, and so $h_{\PP^{4494}}(Z,1)=4420$, as desired.  Recall, however, that this result holds for a special case of $Z$, with its double points lying on specific subspaces.  To get the general result, consider a scheme $Z'$ consisting of 52 double points on $\Split_{28}(\PP^3)$ in general position.  By the semicontinuity theorem \cite[Theorem III.12.8]{Hartshorne}, $h^0(\PP^{4494},\sheaf I_{Z'}(1))\leq h^0(\PP^{4494},\sheaf I_Z(1))$.  By Definition \ref{HilbertFunctionDefinition}, this implies that $h_{\PP^{4494}}(Z',1)\geq h_{\PP^{4494}}(Z,1)=4420$, and the result follows.

Now we will prove the general method.  Note that, instead of just double points, we also include other spaces which appear in the splitting induction.

\begin{lemma}\label{aiDifference}
If $d\geq i\ell+1$, then $a_{i+1}(d)+a_i(d-\ell) = \binom{n-\ell+d}{d}+a_i(d)$.
\end{lemma}

\begin{proof}
Let $g_i(d) = \sum_{j=1}^i(-1)^{j-1}\binom{i}{j}\binom{n+d-j\ell}{d-j\ell}$ and $h_i(d)=a_i(d)-g_i(d)$.  Then
\begin{align*}
&g_{i+1}(d)+g_i(d-\ell)\\
&\quad= \sum_{j=1}^{i+1}(-1)^{j-1}\binom{i+1}{j}\binom{n+d-j\ell}{d-j\ell}+\sum_{j=1}^i(-1)^{j-1}\binom{i}{j}\binom{n+d-(j+1)\ell}{d}\\
&\quad= \sum_{j=1}^{i+1}(-1)^{j-1}\binom{i+1}{j}\binom{n+d-j\ell}{d-j\ell}+\sum_{j=2}^{i+1}(-1)^{j-2}\binom{i}{j-1}\binom{n+d-j\ell}{d-j\ell}\\
&\quad= (i+1)\binom{n+d-\ell}{d-\ell}+\sum_{j=2}^{i+1}(-1)^{j-1}\left(\binom{i+1}{j}-\binom{i}{j-1}\right)\binom{n+d-j\ell}{d-j\ell}\\
&\quad= (i+1)\binom{n+d-\ell}{d-\ell}+\sum_{j=2}^{i+1}(-1)^{j-1}\binom{i}{j}\binom{n+d-j\ell}{d-j\ell}\\
&\quad= \binom{n+d-\ell}{d-\ell}+\sum_{j=1}^i(-1)^{j-1}\binom{i}{j}\binom{n+d-j\ell}{d-j\ell}\\
&\quad= \binom{n+d-\ell}{d-\ell}+g_i(d).
\end{align*}

It follows immediately from (\ref{BackwardDifferenceIdentity}) that
\begin{equation*}
h_{i+1}(d)+h_i(d-l) = h_i(d).\qedhere
\end{equation*}
\end{proof}

\begin{proposition}\label{specializationLemma}
Suppose $d\geq (i+1)\ell+1$.  If $\statement A_{i+1}(d)$ and $\statement A_i(d-\ell)$ are both true and subabundant (resp. superabundant), then $\statement A_i(d)$ is true and subabundant (resp. superabundant).
\end{proposition}

\begin{proof}
We construct three subschemes of $\PP R_d$ as follows.

For each $j\in\{1,\ldots,i+1\}$, choose generic $\tuple g_j\in R_1^{\ell}$ and let $P=\bigcup_{j=1}^i\PP\left(\prod \tuple g_j\right)R_{d-\ell}$ and $P_{i+1}=\PP\left(\prod \tuple g_{i+1}\right)R_{d-\ell}$.

Let $Z$ consist of
\begin{itemize}
\item $s(d)$ generic double points on $\Split_d(\PP^n)$ such that, for any $I\subset\{1,\ldots,i+1\}$, $(\nabla^{i-\lvert I\rvert}_\ell s)(d-\lvert I\rvert\ell)$ of the double points lie on $\bigcap_{j\in I}\PP\left(\prod \tuple g_j\right)R_{d-\ell}$,
\item $t(d)$ generic points on $\Split_d(\PP^n)$ such that, for any $I\subset\{1,\ldots,i+1\}$, $(\nabla^{i-\lvert I\rvert}_\ell t)(d-\lvert I\rvert\ell)$ of the points lie on $\bigcap_{j\in I}\PP\left(\prod \tuple g_j\right)R_{d-\ell}$,
\item $u(d)$ generic $\PP^d$'s of the form $\PP\sum_{m=1}^{d+1}\Span\{\pi_m(f'')\}$, $f''\in R_1^{d+1}$, such that, for any $I\subset\{1,\ldots,i+1\}$, $(\nabla^{i-\lvert I\rvert}_\ell u)(d-\lvert I\rvert\ell)$ of the $\PP^d$'s lie on $\bigcap_{j\in I}\PP\left(\prod \tuple g_j\right)R_{d-\ell}$, and
\item $v(d)$ generic $\PP^n$'s of the form $\PP\pi_1(f''')R_1$, $f'''\in R_1^{d}$, such that, for any $I\subset\{1,\ldots,i+1\}$, $(\nabla^{i-\lvert I\rvert}_\ell v)(d-\lvert I\rvert\ell)$ of the double points lie on $\bigcap_{j\in I}\PP\left(\prod \tuple g_j\right)R_{d-\ell}$.
\end{itemize}

By construction, using Lemmas \ref{TangentToSplit} and \ref{DoublePoints}, we see that \begin{align*}
h_{\PP R_d}(Z\cup P,1) &\leq \dim A_i(d)\text{ (by semicontinuity)},\\
h_{\PP R_d}(Z\cup P\cup P_{i+1},1) &= \dim A_{i+1}(d)\text{, and}\\
h_{P_{i+1}}((Z\cup P)\cap P_{i+1},1)&= \dim A_i(d-\ell).
\end{align*}

\textit{Case 1}.  Suppose $\statement A_{i+1}(d)$ and $\statement A_i(d-\ell)$ are both subabundant.  Then we have, by Lemma \ref{HilbertFunctionInequality},
\begin{align*}
\dim A_i(d) &\geq a_{i+1}(d)+a_i(d-\ell)-\binom{n+d-\ell}{d-\ell}\\
&= \binom{n+d-\ell}{d-\ell}+a_i(d)-\binom{n+d-\ell}{d-\ell}\text{ (by Lemma \ref{aiDifference})}\\
&= a_i(d).
\end{align*}

\textit{Case 2}.  Suppose $\statement A_{i+1}(d)$ and $\statement A_i(d-\ell)$ are both superabundant.  Then we have, by Lemma \ref{HilbertFunctionInequality},
\begin{align*}
\dim A_i(d) &\geq\binom{n+d}{d}+\binom{n+d-\ell}{d-\ell}-\binom{n+d-\ell}{d-\ell}\\
&=\binom{n+d}{d}.\qedhere
\end{align*}
\end{proof}

We will need the following combinatorial identity.  (See, for example, \cite[(5.43)]{GKP}.)
\begin{lemma}\label{BinomialCoefficientIdentity}
For any $n\geq 0$,
\begin{equation*}
\sum_{j=0}^n(-1)^j\binom{n}{j}\binom{r-j\ell}{n}=\ell^n.
\end{equation*}
\end{lemma}

\begin{proposition}\label{An}
If $\statement A_n(\ell n+1)$ is true and superabundant, then $\statement A_n(d)$ is true and superabundant for all $d\geq\ell n+1$.
\end{proposition}

\begin{proof}
Consider a generic $\tuple h\in R_1^{d-(\ell n + 1)}$ and let $P'=\PP\left(\prod\tuple h\right)R_{\ell n + 1}$.  Let $P$ and $Z$ be as in the proof of Proposition \ref{specializationLemma} with $i=n$, with the additional condition that $Z$ sits inside $P'$.  Then by construction,
\begin{align*}
h_{\PP R_d}(Z\cup P,1) &\leq \dim A_n(d)\text{ (by semicontinuity)},\\
h_{\PP R_d}(Z\cup P\cup P',1) &= \dim\left(\sum_{j=1}^n\left(\prod \tuple g_j\right)R_{d-\ell}+\left(\prod \tuple h\right)R_{\ell n + 1}\right)\text{, and}\\
h_{P'}((Z\cup P)\cap P',1) &= \dim A_n(\ell n + 1)\\
&= \binom{n+\ell n+1}{\ell n + 1}\text{ (by assumption).}
\end{align*}

Therefore, by Lemma \ref{HilbertFunctionInequality},
\begin{align*}
\dim A_n(d) &\geq \dim\left(\sum_{j=1}^n\left(\prod \tuple g_j\right)R_{d-\ell}+\left(\prod \tuple h\right)R_{\ell n + 1}\right)\\
&\quad+\binom{n+\ell n+1}{\ell n + 1}-\binom{n+\ell n+1}{\ell n + 1}\\
&= \dim \sum_{j=1}^n\left(\prod \tuple g_j\right)R_{d-\ell}+\dim \left(\prod \tuple h\right)R_{\ell n + 1}\\
&\quad-\dim\sum_{j=1}^n\left(\prod \tuple g_j|\tuple h\right)R_{\ell n + 1 - \ell}\\
&= \sum_{j=1}^n(-1)^{j-1}\binom{n}{j}\binom{n+d-j\ell}{d-j\ell}+\binom{n+\ell n + 1}{\ell n + 1}\\
&\quad -\sum_{j=1}^n(-1)^{j-1}\binom{n}{j}\binom{n+\ell n + 1 - j\ell}{\ell n + 1 - j\ell}\\
&=\binom{n+d}{d} - \sum_{j=0}^n(-1)^j\binom{n}{j}\binom{n+d-j\ell}{n}\\
&\quad+\sum_{j=0}^n(-1)^j\binom{n}{j}\binom{n+\ell n + 1 - j\ell}{n}\\
&= \binom{n+d}{d} - \ell^n + \ell^n\text{ (by Lemma \ref{BinomialCoefficientIdentity})}\\
&= \binom{n+d}{d}.\qedhere
\end{align*}
\end{proof}

\begin{theorem}[Restriction induction, fixed dimension]\label{RestrictionInductionFixedDimension}
Suppose $\statement A_n(d)$ is equiabundant for all $d\geq\ell n+1$.  If $\statement A_{\left\lfloor\frac{d-1}{\ell}\right\rfloor}(d)$ is true and subabundant (resp., superabundant) for all $d\leq\ell n + 1$, then $\statement A_0(d)$ is true and subabundant (resp., superabundant) for all $d$.  In particular, if $t=u=v=0$, then $\sigma_{s(d)}(\Split_d(\PP^n))$ is nondefective for all $d$.
\end{theorem}

\begin{proof}
It follows immediately from Proposition \ref{An} that $\statement A_n(d)$ is true for all $d\geq n\ell+1$.  The rest follows by induction using Proposition \ref{specializationLemma}.
\end{proof}

\pagebreak

\section{Restriction Induction with Fixed Degree}
\label{FixedDegree}

In the previous section, we fixed $n$ and found a method for determining the nondefectivity of $\sigma_s(\Split_d(\PP^n))$ by induction on $d$.  However, since many of the arguments rely on properties of the binomial coefficient $\binom{n+d}{d}$, and $\binom{n+d}{d}=\binom{n+d}{n}$, we will see that essentially the same process will allow us to determine the nondefectivity of $\sigma_s(\Split_d(\PP^n))$ by fixing $d$ and using induction on $n$.

As before, consider $n,d,\ell\in\NN$ and functions $s,t,u,v:\NN\rightarrow\ZZ_{\geq 0}$. Fix an $i\in\{0,\ldots,d\}$.  Choose the following tuples of generic linear forms.
\begin{itemize}
\item For each $j\in\{1,\ldots,i\}$, let $V_j$ be a generic subspace of $R_1$ with $\dim V_j=n-\ell+1$.
\item For each $j\in\{1,\ldots,i\}$ and $k\in\{1,\ldots,(\nabla^{i-1}_\ell s)(n-\ell)\}$, let $\tuple f_{j,k}\in V_j^{d}$.
\item For each $j\in\{1,\ldots,i\}$ and $k\in\{1,\ldots,(\nabla^{i-1}_\ell t)(n-\ell)\}$, let $\tuple f'_{j,k}\in V_j^{d+1}$.
\item For each $j\in\{1,\ldots,i\}$ and $k\in\{1,\ldots,(\nabla^{i-1}_\ell u)(n-\ell)\}$, let $\tuple f''_{j,k}\in V_j^{d+1}$.
\item For each $j\in\{1,\ldots,i\}$ and $k\in\{1,\ldots,(\nabla^{i-1}_\ell v)(n-\ell)\}$, let $\tuple f'''_{j,k}\in V_j^{d}$.
\item For each $j\in\{1,\ldots,(\nabla^i_\ell s )(n)\}$ let $\tuple f_j\in R_1^{d}$.
\item For each $j\in\{1,\ldots,(\nabla^i_\ell t)(n)\}$ let $\tuple f'_j\in R_1^{d+1}$.
\item For each $j\in\{1,\ldots,(\nabla^i_\ell u)(n)\}$ let $\tuple f''_j\in R_1^{d+1}$.
\item For each $j\in\{1,\ldots,(\nabla^i_\ell v)(n)\}$ let $\tuple f'''_j\in R_1^{d}$.
\end{itemize}

We then define the following subspace of $R_d$:

\begin{align*}
&B_i(n,d,\ell,s,t,u,v) =\quad\sum_{j=1}^i S_dV_j\\
&\quad+\sum_{j=1}^i\sum_{k=1}^{(\nabla^{i-1}_\ell s)(n-\ell)}\sum_{m=1}^d\pi_m(\tuple f_{j,k})R_1+\sum_{j=1}^i\sum_{k=1}^{(\nabla^{i-1}_\ell t)(n-\ell)}\Span\{\pi_1(\tuple f'_{j,k})\}\\
&\quad+\sum_{j=1}^i\sum_{k=1}^{(\nabla^{i-1}_\ell u)(n-\ell)}\sum_{m=1}^{d+1}\Span\{\pi_m(\tuple f''_{j,k})\}+\sum_{j=1}^i\sum_{k=1}^{(\nabla^{i-1}_\ell v)(n-\ell)}\pi_1(\tuple f'''_{j,k})R_1\\
&\quad+\sum_{j=1}^{(\nabla^i_\ell s )(n)}\sum_{m=1}^d\pi_m(\tuple f_j)R_1+\sum_{j=1}^{(\nabla^i_\ell t )(n)}\Span\{\pi_1(\tuple f'_j)\}\\
&\quad+\sum_{j=1}^{(\nabla^i_\ell u )(n)}\sum_{m=1}^{d+1}\Span\{\pi_m(\tuple f''_j)\}+\sum_{j=1}^{(\nabla^i_\ell v )(n)}\pi_1(\tuple f'''_j)R_1.
\end{align*}

We may abbreviate $B_i(n,d,\ell,s,t,u,v)$ as $B_i(n)$ if desired.

We define the function
\begin{align*}
b_i(n,d,\ell,s,t,u,v) &= \sum_{j=1}^i(-1)^{j-1}\binom{i}{j}\binom{n+d-j\ell}{n-j\ell}+i\ell n(\nabla^{i-1}_\ell s)(n-\ell)\\
&\quad + i\ell(\nabla^{i-1}_\ell u)(n-\ell)+(dn+1)(\nabla^i_\ell s)(n)+(\nabla^i_\ell t)(n)\\
&\quad + (d+1)(\nabla^i_\ell u)(n) + (n+1)(\nabla^i_\ell v)(n).
\end{align*}

Again, we may abbreviate this as $b_i(n)$.  As in Lemma \ref{AiExpdim},
\begin{equation*}
\dim B_i(n)\leq\min\left\{b_i(n),\binom{n+d}{d}\right\}
\end{equation*}

We define the statement $\statement B_i(n,d,\ell,s,t,u,v)$, or $\statement B_i(n)$, analogously to $\statement A_i(d)$.  Then, proceeding almost exactly as in Section \ref{RestrictionInduction}, we finally obtain the following result.

\begin{theorem}[Restriction induction, fixed degree]\label{RestrictionInductionFixedDegree}
Suppose $\statement B_d(n)$ is equiabundant for all $n\geq\ell d+1$.  If $\statement B_{\left\lfloor\frac{n-1}{\ell}\right\rfloor}(n)$ is true and subabundant (resp., superabundant) for all $n\leq\ell d + 1$, then $\statement B_0(n)$ is true and subabundant (resp., superabundant) for all $n$.  In particular, if $t=u=v=0$, then $\sigma_{s(n)}(\Split_d(\PP^n))$ is nondefective for all $n$.
\end{theorem}

\pagebreak

\chapter{Results}

\section{Improved result for $n=3$}

In this section, we apply restriction induction to get a small improvement in the bounds for $s$ in Theorem \ref{AboResult}.

\begin{definition}
For any $n,k\in\ZZ_{\geq 0}$, the \define{Stirling number of the second kind} is
\begin{equation*}
\stirlingtwo{n}{k}=\frac{1}{k!}\sum_{j=0}^n(-1)^{k-j}\binom{k}{j}j^n.
\end{equation*}
\end{definition}

For a discussion of Stirling numbers of the second kind, see for example \cite[Section 6.1]{GKP}.  We will use the following two properties.
\begin{itemize}
\item If $n<k$, then $\stirlingtwo{n}{k}=0$.
\item If $n=k$, then $\stirlingtwo{n}{k}=1$.
\end{itemize}

If $f$ is a polynomial and $g$ is a monomial, then $[g]f$ is the coefficient of $g$ in $f$.

The following result, which relies on Stirling numbers of the second kind, will prove useful.

\begin{lemma}\label{AnEquiabundant}
Suppose, for each $r\in\{0,\ldots,\ell-1\}$, there exist polynomial functions $s_r$, $t_r$, $u_r$, and $v_r$ which are either zero or of degree $n-1$ such that $s(d)=s_r(d)$, $t(d)=t_r(d)$, $u(d)=u_r(d)$, and $v(d)=v_r(d)$ for all $d\equiv r\pmod\ell$.  If 
\begin{equation*}
n!\left(n\left(\left[d^{n-1}\right]s_r(d)\right)+\left(\left[d^{n-1}\right]u_r(d)\right)\right)=1,
\end{equation*}
then $\statement A_n(d)$ is equiabundant for all $d\geq\ell n + 1$ and $\statement B_d(n)$ is equiabundant for all $n\geq d\ell+1$.
\end{lemma}

\begin{proof}
Let $f$ be any polynomial of degree $n-1$.  By induction on $i$, we have
\begin{align*}
(\nabla^i_\ell f)(d) &= (\nabla^{i-1}_\ell f)(d)-(\nabla^{i-1}_\ell f)(d-\ell)\\
&= \sum_{j=0}^{i-1}(-1)^j\binom{i-1}{j}f(d-j\ell)-\sum_{j=0}^{i-1}(-1)^j\binom{i-1}{j}f(d-(j+1)\ell)\\
&= \sum_{j=0}^{i-1}(-1)^j\binom{i-1}{j}f(d-j\ell)+\sum_{j=1}^{i}(-1)^{j-1+1}\binom{i-1}{j-1}f(d-j\ell)\\
&= f(d) + \sum_{j=1}^{i-1}(-1)^j\left[\binom{i-1}{j}+\binom{i-1}{j-1}\right]f(d-j\ell)+(-1)^if(d-i\ell)\\
&= \sum_{j=0}^i(-1)^j\binom{i}{j}f(d-j\ell)\\
&= \sum_{j=0}^i(-1)^j\binom{i}{j}\sum_{k=0}^{n-1}([d^k]f(d))(d-j\ell)^k\\
&=\sum_{j=0}^i(-1)^j\binom{i}{j}\sum_{k=0}^{n-1}([d^k]f(d))\sum_{m=0}^k\binom{k}{m}d^m(-j\ell)^{k-m}\\
&=\sum_{k=0}^{n-1}\sum_{m=0}^k([d^k]f(d))\binom{k}{m}d^m(-\ell)^{k-m}\sum_{j=0}^i(-1)^j\binom{i}{j}j^{k-m}\\
&=\sum_{k=0}^{n-1}\sum_{m=0}^k([d^k]f(d))\binom{k}{m}d^m(-\ell)^{k-m}(-1)^ki!\stirlingtwo{k-m}{i}\\
&=\sum_{k=0}^{n-1}\sum_{m=0}^{k-i}([d^k]f(d))\binom{k}{m}d^m(-\ell)^{k-m}(-1)^ki!\stirlingtwo{k-m}{i}\\
&=\sum_{m=0}^{n-1-i}\left[\sum_{k=m+i}^{n-1}(-1)^m([d^k]f(d))\ell^{k-m}i!\binom{k}{m}\stirlingtwo{k-m}{i}\right]d^m.
\end{align*}

In particular, $(\nabla^{n-1}_\ell f)(d)=(n-1)!\left(\left[d^{n-1}\right]f(d)\right)\ell^{n-1}$ and $(\nabla^n_\ell f)(d)=0$.

Therefore,
\begin{align*}
&a_n(d)=\\
&\quad\sum_{j=1}^n(-1)^{j-1}\binom{n}{j}\binom{n+d-j\ell}{d-j\ell}+n\ell n(n-1)!\left(\left[d^{n-1}\right]s_r(d)\right)\ell^{n-1}\\
&\quad+ n\ell(n-1)!\left(\left[d^{n-1}\right]u_r(d)\right)\ell^{n-1}\\
&=-\sum_{j=0}^n(-1)^j\binom{n}{j}\binom{n+d-j\ell}{n}+\binom{n+d}{d}\\
&\quad+\ell^n\left(n!\left(n\left(\left[d^{n-1}\right]s_r(d)\right)+\left(\left[d^{n-1}\right]u_r(d)\right)\right)\right)\\
&= -\ell^n+\binom{n+d}{d}+\ell^n\text{ (by Lemma \ref{BinomialCoefficientIdentity})}\\
&= \binom{n+d}{d}.
\end{align*}

The same argument holds for $b_d(n)$.
\end{proof}

Consider the following functions on $\mathbb N$:
\begin{align*}
s_1(d) =
\begin{cases}
\frac{1}{18}d^2+\frac{5}{18}d&\text{ if }d\equiv 0,4\pmod{9}\\
\frac{1}{18}d^2+\frac{5}{18}d+\frac{2}{9}&\text{ if }d\equiv 2,5,8\pmod{9}\\
\frac{1}{18}d^2+\frac{5}{18}d+\frac{2}{3}&\text{ if }d\equiv 1,3\pmod{9}\\
\frac{1}{18}d^2+\frac{5}{18}d+\frac{1}{3}&\text{ if }d\equiv 6,7\pmod{9}\\
\end{cases}\\
s''_2(d) =
\begin{cases}
\frac{1}{18}d^2+\frac{7}{18}d&\text{ if }d\equiv 0\pmod{9}\\
\frac{1}{18}d^2+\frac{7}{18}d+\frac{5}{9}&\text{ if }d\equiv 1,4,7\pmod{9}\\
\frac{1}{18}d^2+\frac{7}{18}d+1&\text{ if }d\equiv 2\pmod{9}\\
\frac{1}{18}d^2+\frac{7}{18}d+\frac{1}{3}&\text{ if }d\equiv 3,8\pmod{9}\\
\frac{1}{18}d^2+\frac{7}{18}d+\frac{2}{3}&\text{ if }d\equiv 5,6\pmod{9}.\\
\end{cases}
\end{align*}
Note that $s_1(d) = \left\lfloor\frac{\binom{d+3}{3}}{3d+1}\right\rfloor$ and $s''_2(d) = \left\lceil\frac{\binom{d+3}{3}}{3d+1}\right\rceil$ for all $d\in\{1,\ldots,9\}$.  Since $s_1(d)-s_1(d-9)=d-2$  and $s''_2(d)-s''_2(d-9)=d-1$ for all $d\geq 10$, we see that $s_1(d),s''_2(d)\in\mathbb N$ for all $d\in\mathbb N$.

Recall the functions $s'_1$ and $s'_2$ from Theorem \ref{AboResult}.  Note that $s_1(d)\geq s'_1(d)$ for all $d$ and $s''_2(d)\leq s'_2(d)$ for all $d\not\equiv 0\pmod 3$.

The following result gives some insight on why these functions were chosen.

\begin{lemma}
If $d\in\mathbb N$, then $s_1(d)\leq\frac{\binom{d+3}{3}}{3d+1}\leq s''_2(d)$.
\end{lemma}

\begin{proof}
By definition, the result is trivial for $d\leq 9$.  

Note that
\begin{equation*}
\frac{\binom{d+3}{3}}{3d+1}-\frac{\binom{d-9+3}{3}}{3(d-9)+1} = d - \frac{5}{3} -\frac{40}{3}\frac{1}{9d^2-75d-26}.
\end{equation*}

Assume $d\geq 10$.  By induction on $d$, we have
\begin{equation*}
d-2 \leq d - \frac{5}{3} -\frac{40}{3}\frac{1}{9d^2-75d-26}  \leq d-1
\end{equation*}
\begin{equation*}
s_1(d-9)+d-2 \leq \frac{\binom{d-9+3}{3}}{3(d-9)+1}+\frac{\binom{d+3}{3}}{3d+1}-\frac{\binom{d-9+3}{3}}{3(d-9)+1}  \leq s''_2(d-9) + d - 1
\end{equation*}
\begin{equation*}
s_1(d) \leq \frac{\binom{d+3}{3}}{3d+1}  \leq s''_2(d).
\end{equation*}
\end{proof}

Note that
\begin{equation*}
3!\left(3\cdot\frac{1}{18}+0\right)=1
\end{equation*}
and so, by Lemma \ref{AnEquiabundant}, $\statement A_3(3,d,9,s''_i,0,0,0)$ is equiabundant for all $n,d\geq 28$ and $i=1,2$.

By Macaulay2 Computation \ref{NEqualsThreeM2}, we may apply Theorem \ref{RestrictionInductionFixedDimension} to see that $\sigma_{s_1(d)}(\Split_d(\PP^3))$ and $\sigma_{s''_2(d)}(\Split_d(\PP^3))$ are nondefective for all $d$.  
Combining this with Corollary \ref{TerraciniCorollary} and Theorem \ref{AboResult}, we obtain the following result.

\begin{theorem}\label{nEqualsThreeResult}
Let 
\begin{align*}
s_1(d) =
\begin{cases}
\frac{1}{18}d^2+\frac{5}{18}d&\text{ if }d\equiv 0,4\pmod{9}\\
\frac{1}{18}d^2+\frac{5}{18}d+\frac{2}{9}&\text{ if }d\equiv 2,5,8\pmod{9}\\
\frac{1}{18}d^2+\frac{5}{18}d+\frac{2}{3}&\text{ if }d\equiv 1,3\pmod{9}\\
\frac{1}{18}d^2+\frac{5}{18}d+\frac{1}{3}&\text{ if }d\equiv 6,7\pmod{9}\\
\end{cases}\\
s_2(d) =
\begin{cases}
\frac{1}{18}d^2+\frac{1}{3}d+1 &\text{ if }d\equiv 0\pmod{6}\\
\frac{1}{18}d^2+\frac{1}{3}d+\frac{1}{2} &\text{ if }d\equiv 3\pmod{6}\\
\frac{1}{18}d^2+\frac{7}{18}d+\frac{5}{9}&\text{ if }d\equiv 1,4,7\pmod{9}\\
\frac{1}{18}d^2+\frac{7}{18}d+1&\text{ if }d\equiv 2\pmod{9}\\
\frac{1}{18}d^2+\frac{7}{18}d+\frac{2}{3}&\text{ if }d\equiv 5\pmod{9}\\
\frac{1}{18}d^2+\frac{7}{18}d+\frac{1}{3}&\text{ if }d\equiv 8\pmod{9}.\\
\end{cases}
\end{align*}
If $s\leq s_1(d)$ or $s\geq s_2(d)$, then $\sigma_s(\Split_d(\PP^3))$ is nondefective for all $d\in\NN$.  
\end{theorem}

\pagebreak

\section{A result for $n\geq 4$}
\label{NAtLeastFour}

Consider the function
\begin{equation*}
\tilde s(d) = \begin{cases}
\frac{1}{24}d^2+\frac{1}{12}d&\text{ if }d\equiv 0,4\pmod{6}\\
\frac{1}{24}d^2+\frac{1}{6}d-\frac{5}{24}&\text{ if }d\equiv 1\pmod{6}\\
\frac{1}{24}d^2+\frac{1}{12}d-\frac{1}{3}&\text{ if }d\equiv 2\pmod{6}\\
\frac{1}{24}d^2+\frac{1}{6}d+\frac{1}{8}&\text{ if }d\equiv 3,5\pmod{6}.\\
\end{cases}
\end{equation*}

Note that $\tilde s(d)\leq\frac{1}{24}d^2+\frac{1}{6}d+\frac{1}{8}$ for all $d\geq 1$.  Then
\begin{align*}
a(3,d,\tilde s,\tilde s,\tilde s,\tilde s) &= (3d+1)\tilde s(d)+\tilde s(d)+(d+1)\tilde s(d)+4\tilde s(d)\\
&=(4d+7)\tilde s(d)\\
&\leq(4d+7)\left(\frac{1}{24}d^2+\frac{1}{6}d+\frac{1}{8}\right)\\
&=\frac{1}{6}d^3+\frac{23}{24}d^2+\frac{5}{3}d+\frac{7}{8}\\
&<\frac{1}{6}d^3+d^2+\frac{11}{6}d+1=\binom{d+3}{d}.
\end{align*}
and so $\statement A(3,d,\tilde s,\tilde s,\tilde s,\tilde s)$ is subabundant.

Note that
\begin{equation*}
3!\left(3\cdot\frac{1}{24}+\frac{1}{24}\right)=1
\end{equation*}
and so, by Lemma \ref{AnEquiabundant}, $\statement A_3(3,d,6,\tilde s,\tilde s,\tilde s,\tilde s)$ is equiabundant for $d\geq 18$.  By Macaulay2 Computation \ref{TildeSM2}, Theorem \ref{RestrictionInductionFixedDimension}, and an argument similar to Corollary \ref{TerraciniCorollary}, we obtain the following result.

\begin{proposition}\label{TildeSBound}
For any $d\in\NN$, $\statement A(3,d,s,t,u,v)$ is true and subabundant for all functions $s,t,u,v$ such that $s(d),t(d),u(d),v(d)\leq\tilde s(d)$.
\end{proposition}

Suppose $n_0\geq 3$ and let $\mathcal A_{n_0}$ be the set of all statements of the form $\statement A(n,d,s,t,u,v)$ such that $n\geq n_0$ and $s(d)$, $t(d)$, $u(d)$, and $v(d)$ are integers divisible by $2^{n-3}$.  We define the following maps from $\mathcal A_4\rightarrow\mathcal A_3$:
\begin{align*}
\varphi_0(\statement A(n,d,s,t,u,v))&=\statement A(n-1,d,s/2,(t+u)/2,u/2,(s+v)/2),\\
\varphi_1(\statement A(n,d,s,t,u,v))&=\statement A(n-1,d-1,s/2,(t+v)/2,(s+u)/2,v/2),\text{ and}\\
\varphi_2(\statement A(n,d,s,t,u,v))&=\statement A(n-1,d-2,0,v/2,s/2,0).
\end{align*}
Note that, by Theorem \ref{SplittingInductionTheorem}, if $\statement A\in\mathcal A_4$ and $\varphi_0(\statement A)$, $\varphi_1(\statement A)$, and $\varphi_2(\statement A)$ are all true and subabundant, then $\statement A$ is true and subabundant.

\begin{lemma}\label{commutative}
For any distinct $i,j\in\{0,1,2\}$, the following diagram commutes.
\begin{equation*}
\begin{CD}
\mathcal A_5 @>\varphi_i>> \mathcal A_4\\
@VV\varphi_jV            @VV \varphi_jV\\
\mathcal A_4 @>\varphi_i>> \mathcal A_3\\
\end{CD}
\end{equation*}
\end{lemma}

\begin{proof}
Suppose $n\geq 5$.  By definition,
\begin{align*}
&\varphi_0(\varphi_1(\statement A(n,d,s,t,u,v))=\varphi_1(\varphi_0(\statement A(n,d,s,t,u,v))\\
&\quad=\statement A(n-2,d-1,s/4,(s+t+u+v)/4,(s+u)/4,(s+v)/4)\\
&\varphi_0(\varphi_2(\statement A(n,d,s,t,u,v))=\varphi_2(\varphi_0(\statement A(n,d,s,t,u,v))\\
&\quad=\statement A(n-2,d-2,0,(s+v)/4,s/4,0)\\
&\varphi_1(\varphi_2(\statement A(n,d,s,t,u,v))=\varphi_2(\varphi_1(\statement A(n,d,s,t,u,v))\\
&\quad=\statement A(n-2,d-3,0,v/4,s/4,0).\qedhere
\end{align*}
\end{proof}
For any $i,j,k$ such that $i+j+k\leq n-3$, define the function $\varphi_{i,j,k}=\varphi_0^i\circ \varphi_1^j\circ \varphi_2^k$, where the exponents represent function composition, which, as we just saw in Lemma \ref{commutative}, is commutative.

\begin{lemma}\label{VertexFormulas}
Let $c\in\NN$.  If $i+j+k\leq n-3$ and $k\leq 2$, then
\begin{align*}
&\varphi_{i,j,k}(\statement A(n,d,2^{n-3}c,0,0,0))=\\
&\quad\begin{cases}
\statement A(n-i-j,d-j,2^{n-3-i-j}c,2^{n-3-i-j}cij,2^{n-3-i-j}cj,2^{n-3-i-j}ci) & \text{ if }k=0\\
\statement A(n-i-j-1,d-j-2,0,2^{n-4-i-j}ci,2^{n-4-i-j}c,0)& \text{ if }k=1\\
\statement A(n-i-j-2,d-j-4,0,0,0,0)&\text{ if }k=2
\end{cases}
\end{align*}
\end{lemma}

\begin{proof}
First, we check the base cases for each $k$.
\begin{align*}
\varphi_{0,0,0}(\statement A(n,d,2^{n-3}c,0,0,0))&=\statement A(n,d,2^{n-3}c,0,0,0)\\
\varphi_{0,0,1}(\statement A(n,d,2^{n-3}c,0,0,0))&=\statement A(n-1,d-2,0,0,2^{n-4}c,0)\\
\varphi_{0,0,2}(\statement A(n,d,2^{n-3}c,0,0,0))&=\statement A(n-2,d-4,0,0,0,0).
\end{align*}
Next, we use induction on $i$.
\begin{align*}
&\varphi_0(\statement A(n-(i-1),d,2^{n-3-(i-1)}c,0,0,2^{n-3-(i-1)}c(i-1)))\\
&\quad=\statement A(n-i,d,2^{n-3-i}c,0,0,2^{n-3-i}ci)\\
&\varphi_0(\statement A(n-(i-1)-1,d-2,0,2^{n-4-(i-1)}c(i-1),2^{n-4-(i-1)}c,0)\\
&\quad=\statement A(n-i-1,d-2,0,2^{n-4-i}ci,2^{n-4-i}c,0)\\
&\varphi_0(\statement A(n-(i-1)-2,d-4,0,0,0,0)\\
&\quad=\statement A(n-i-2,d-4,0,0,0,0).
\end{align*}
Finally, we use induction on $j$.
\begin{align*}
&\varphi_1(\statement A(n-i-(j-1),d-(j-1),2^{n-3-i-(j-1)}c,2^{n-3-i-(j-1)}ci(j-1),\\
&\quad\quad 2^{n-3-i-(j-1)}c(j-1),2^{n-3-i-(j-1)}ci))\\
&\quad=\statement A(n-i-j,d-j,2^{n-3-i-j}c,2^{n-3-i-j}cij,2^{n-3-i-j}cj,2^{n-3-i-j}ci)\\
&\varphi_1(\statement A(n-i-(j-1)-1,d-(j-1)-2,0,2^{n-4-i-(j-1)}ci,2^{n-4-i-(j-1)}c,0))\\
&\quad=\statement A(n-i-j-1,d-j-2,0,2^{n-4-i-j}ci,2^{n-4-i-j}c,0)\\
&\varphi_1(\statement A(n-i-(j-1)-2,d-(j-1)-4,0,0,0,0))\\
&\quad=\statement A(n-i-j-2,d-j-4,0,0,0,0).\qedhere
\end{align*}
\end{proof}

Consider a statement $\statement A=\statement A(n,d,2^{n-3}c,0,0,0)$ with $n\geq 4$ and $d\geq n$.  The \define{splitting graph} $\Gamma(\statement A)$ is the directed acyclic graph with vertex set
\begin{equation*}
V=\{\varphi_{i,j,k}(\statement A):i+j+k\leq n-3\text{ and }k\leq 2\}
\end{equation*}
and edge set
\begin{equation*}
E=\{\statement A'\statement A'':\varphi_i(\statement A')=\statement A''\text{ for some }i\in\{0,1,2\}\}.
\end{equation*}
For example, $\Gamma(\statement A(5,17,24,0,0,0))$ is shown in Figure \ref{SplittingGraphExample}.

\begin{figure}[h]
\centering
\begin{tikzpicture}[>=latex,scale=3]
\node (1) at (0,0) [draw,ellipse] {$\statement A(5,17,24,0,0,0)$};
\node (2) at (0,1) [draw,ellipse] {$\statement A(4,17,12,0,0,12)$};
\node (3) at (-.866,.5) [draw,ellipse] {$\statement A(3,16,6,6,6,6)$};
\node (4) at (-.866,-.5) [draw,ellipse] {$\statement A(4,16,12,0,12,0)$};
\node (5) at (0,-1) [draw,ellipse] {$\statement A(3,14,0,0,6,0)$};
\node (6) at (.866,-.5) [draw,ellipse] {$\statement A(4,15,0,0,12,0)$};
\node (7) at (.866,.5) [draw,ellipse] {$\statement A(3,15,0,6,6,0)$};
\node (8) at (0,2) [draw,ellipse] {$\statement A(3,17,6,0,0,12)$};
\node (9) at (-1.732,-1) [draw,ellipse] {$\statement A(3,15,6,0,12,0)$};
\node (10) at (1.732,-1) [draw,ellipse] {$\statement A(3,13,0,0,0,0)$};
\draw[->] (1) -- (2);
\draw[->] (1) -- (4);
\draw[->] (1) -- (6);
\draw[->] (2) -- (7);
\draw[->] (2) -- (3);
\draw[->] (2) -- (8);
\draw[->] (4) -- (3);
\draw[->] (4) -- (5);
\draw[->] (4) -- (9);
\draw[->] (6) -- (5);
\draw[->] (6) -- (7);
\draw[->] (6) -- (10);
\end{tikzpicture}
\caption{Splitting graph of $\statement A(5,17,24,0,0,0)$\label{SplittingGraphExample}}
\end{figure}

A \define{sink vertex} of a directed graph is a vertex with no outgoing edges.  By Theorem~\ref{SplittingInductionTheorem} and induction on the vertices of $\Gamma(\statement A)$, we have the following result.

\begin{proposition}\label{SinkVerticesTrueSubabundant}
If all of the sink vertices of $\Gamma(\statement A)$ are true and subabundant, then $\statement A$ is true and subabundant.
\end{proposition}

Note that, by construction, each sink vertex either has $n=3$, and so we may use Proposition \ref{TildeSBound} to determine its truth and subabundance, or has $s=t=u=v=0$, and thus is certainly true and subabundant.

For example, refer to Figure \ref{SplittingGraphExample}.  Since $\tilde s(14)=9$, $\tilde s(15)=\tilde s(16)=12$, and $\tilde s(17)=15$, we see by Proposition \ref{TildeSBound} that each sink vertex is true and subabundant.  Therefore, by Proposition \ref{SinkVerticesTrueSubabundant}, $\statement A(5,17,24,0,0,0)$ is true, and consequently $\sigma_{24}(\Split_{17}(\PP^5))$ is nondefective.

We now generalize this idea.  

Consider a splitting graph $G=\Gamma(\statement A(n,d,2^{n-3}c,0,0,0))$.  For each $m\in\{0,\ldots,n-2\}$, let $\mathcal V_m$ be the set of all sink vertices of $G$ of the form $\statement A(3,d-m,s,t,u,v)$.  For any such $\statement A\in\mathcal V_m$, let $h(\statement A)=\max\{s(d-m),t(d-m),u(d-m),v(d-m)\}$.  By Lemma~\ref{VertexFormulas}, $c$ will divide $h(\statement A)$.  We define 
\begin{equation*}
g_n(m)=\max\left\{\frac{1}{c}h(\statement A):\statement A\in\mathcal V_m\right\}.
\end{equation*}

\begin{lemma}
If $n\geq 4$ and $0\leq m\leq n-2$, then
\begin{equation*}
g_n(m)=\begin{cases}
n-3 &\text{ if }m=0\text{ or }m=n-3\\
n-4 &\text{ if }n\geq 5\text{ and }m=1\\
1 &\text{ if }m=n-2\\
m(n-m-3) &\text{ if }2\leq m\leq n-4.\\
\end{cases}
\end{equation*}
\end{lemma}

\begin{proof}
We refer to Lemma \ref{VertexFormulas} and use these formulas to find the vertices where $n-i-j-k=3$.  We may ignore the $k=2$ case, as all of the values we are trying to maximize are zero.

\textit{Case 1.}  Suppose $k=0$.  Then $d-j=d-m$, and so $j=m$.  Consequently, $n-i-m=3$, and so $i=n-m-3$ and $m\leq n-3$.  Note that, in this case, $n-3-i-j=n-3-(n-m-3)-m=0$.  Therefore, these sink vertices are of the form $\statement A(3,d-m,c,cm(n-m-3),cm,c(n-m-3))$.

\textit{Case 2.}  Suppose $k=1$.  Then $d-j-2=d-m$, and so $j=m-2$ and $m\geq 2$.  Consequently, $n-i-(m-2)-1=3$, and so $i=n-m-2$ and $m\leq n-2$.  Note that, in this case, $n-4-i-j=n-4-(n-m-2)-(m-2)=0$.  Therefore, these sink vertices are of the form $\statement A(3,d-m,0,c(n-m-2),c,0)$.

We now investigate the maximum values for specific values of $m$.

\textit{Case (a).}  Suppose $m=0$.  Then we must have $k=0$.  These sink vertices are of the form $\statement A(3,d,c,0,0,c(n-3))$.  Therefore, $g_n(0)=n-3$.

\textit{Case (b).}  Suppose $m=1$.  Then we must have $k=0$.  These sink vertices are of the form $\statement A(3,d-1,c,c(n-4),c,c(n-4))$.  Therefore, $g_n(1)=n-4$ if $n\geq 5$ and $g_4(1)=1$.

Note that since $m\leq n-2$, we have completed the $n=4$ case.

\textit{Case (c).}  Suppose $n\geq 5$ and $m=n-3$.  If $k=0$, then the sink vertices are of the form $\statement A(3,d-n+3,c,0,c(n-3),0)$.  If $k=1$, then the sink vertices are of the form $\statement A(3,d-n+3,0,c,c,0)$.  Therefore, $g_n(m)=n-3$.

\textit{Case (d).}  Suppose $m=n-2$.  Then we must have $k=1$.  These sink vertices are of the form $\statement A(3,d-n+2,0,0,c,0)$.  Therefore, $g_n(n-2)=1$.

\textit{Case (e).}  For the remaining cases, we have $n\geq 6$ and $2\leq m\leq n-4$.  In this case, $\max\{1,m(n-m-3),m,n-m-3\}=m(n-m-3)$ and $\max\{n-m-2,1\}=n-m-2$, and so we have $g_n(m)=\max\{m(n-m-3),n-m-2\}$.  Since $m\leq n-4$, we have
\begin{align*}
n-m-3 &\geq 1\\
2(n-m-3) &\geq n-m-2\\
m(n-m-3) &\geq n-m-2\text{ (since $m\geq 2$)}.
\end{align*}
Consequently, $g_n(m)=m(n-m-3)$.
\end{proof}

\begin{thmn}[\ref{ExponentialBound}]
Consider the function
\begin{equation*}
c(n,d)=\min\left\{\left\lfloor\frac{\tilde s(d-m)}{g_n(m)}\right\rfloor:0\leq m\leq n-2\right\}.
\end{equation*}
If $d\geq n\geq 4$ and $s\leq 2^{n-3}c(n,d)$, then $\sigma_s(\Split_d(\PP^n))$ is nondefective.
\end{thmn}

\begin{proof}
Consider the nontrivial sink vertices of $\Gamma(\statement A(n,d,2^{n-3}c(n,d),0,0,0))$.  By construction, they will be of the form $\statement A(3,d-m,s,t,u,v)$, where 
\begin{align*}
\max\{s(d-m),t(d-m),u(d-m),v(d-m)\}&\leq c(n,d)g_n(m)\\
&\leq\frac{\tilde s(d-m)}{g_n(m)}g_n(m)\\
&=\tilde s(d-m).
\end{align*}
By Proposition \ref{TildeSBound}, each sink vertex is true and subabundant, and thus the result follows by  Proposition \ref{SinkVerticesTrueSubabundant}.
\end{proof}

\pagebreak

\section{Results for small $s$}
\label{SmallS}

Note that, by Propositions \ref{LargeN} and \ref{DEqualsTwo}, the dimensions of $\sigma_2(\Split_d(\PP^n))$ are completely known.  Our goal in this section is to prove conjecture \ref{Conjecture} for other small $s$.

We will need the following lemma.

\begin{lemma}\label{F true}
If $\statement A(n,d,s,0,0,0)$ is true and subabundant, then $\statement A(n,d-1,0,0,s,0)$ is true and subabundant.
\end{lemma}

\begin{proof}
For each $j\in\{1,\ldots,s(d)\}$, choose generic $\tuple f_j\in R_1^d$, and let $V_j=\sum_{m=1}^d\pi_m(\tuple f_j)R_1$.  

Note that, for each $j$, $\dim V_j\leq dn+1$.  Indeed, it is spanned by $d(n+1)$ polynomials, but $\prod\tuple f_j$ appears $d$ times, and $d(n+1)-(d-1)=dn+1$.  

However, by assumption, we have $\dim\sum_{j=1}^{s(d)}V_j=s(d)(dn+1)$, and so we must have $\dim V_j=dn+1$.  Furthermore, $\sum_{j=1}^{s(d)}V_j=\bigoplus_{j=1}^{s(d)}V_j$.

Suppose there exist $a_{j,m}\in\kk$ such that $\sum_{j=1}^{s(d)}\sum_{m=1}^da_{j,m}\pi_m(\tuple f_j)=0$.  Therefore, for nonzero $g\in R_1$, we have $\sum_{j=1}^{s(d)}\sum_{m=1}^da_{j,m}\pi_m(\tuple f_j)g=0$.  However, $\sum_{m=1}^da_{j,m}\pi_m(\tuple f_j)g\in V_j$ for each $j$, and this implies that we must have $\sum_{m=1}^da_{j,m}\pi_m(\tuple f_j)g=0$, because, as shown above, the sum of these vectors spaces is a direct sum.  Consequently, $\sum_{j=1}^da_{j,m}\pi_m(\tuple f_j)=0$, for each $i$.

Suppose $a_{j,m}\neq 0$ for some $j,m$.  Without loss of generality, suppose $m=1$.  Then we have
\begin{align*}
\pi_1(\tuple f_j)&=-\sum_{m=2}^d\frac{a_{j,m}}{a_{j,1}}\pi_m(\tuple f_j)\\
&=-(\tuple f_j)_1\sum_{m=2}^d\frac{a_{j,m}}{a_{j,1}}\pi_m(\rho_1(\tuple f_j)).
\end{align*}
Therefore, $(\tuple f_j)_1$ occurs twice in $\tuple f_j$, up to scalar multiplication.  This is a contradiction since $\tuple f_j$ was chosen to be generic.  Therefore, $a_{j,m}=0$ for all $i,j$, and thus $\{\pi_m(\tuple f_j):1\leq j\leq s(d),\,1\leq m\leq d\}$ is a linearly independent set.  Since this set spans $A(n,d-1,0,0,s,0)$ and $a(n,d-1,0,0,s,0)=ds(d)$, the corresponding statement is true.
\end{proof}

\begin{theorem}\label{induction on n}
If $s(dn_0+1)\leq\binom{n_0+d}{d}$ and $\sigma_s(\Split_d(\PP^{n_0}))$ is nondefective, then $\sigma_s(\Split_d(\PP^n))$ is nondefective for all $n\geq n_0$.
\end{theorem}

\begin{proof}
By assumption, $\statement A(n_0,d,s,0,0,0)$ is true and subabundant.  Therefore, by Lemma \ref{F true}, $\statement A(n_0,d-1,0,0,s,0)$ is true and subabundant.  It follows by Theorem \ref{SplittingInductionTheorem} that $\statement A(n_0+1,d,s,0,0,0)$ is true and subabundant.  The result follows by induction on $n$.
\end{proof}

Recall the following function from Theorem \ref{nEqualsThreeResult}.
\begin{equation*}
s_1(d) =
\begin{cases}
\frac{1}{18}d^2+\frac{5}{18}d&\text{ if }d\equiv 0,4\pmod{9}\\
\frac{1}{18}d^2+\frac{5}{18}d+\frac{2}{9}&\text{ if }d\equiv 2,5,8\pmod{9}\\
\frac{1}{18}d^2+\frac{5}{18}d+\frac{2}{3}&\text{ if }d\equiv 1,3\pmod{9}\\
\frac{1}{18}d^2+\frac{5}{18}d+\frac{1}{3}&\text{ if }d\equiv 6,7\pmod{9}\\
\end{cases}
\end{equation*}
The following result follows immediately from Theorem \ref{nEqualsThreeResult} and Lemma \ref{induction on n}.

\begin{corollary}\label{S1Generalized}
If $d\in\NN$ and $s\leq s_1(d)$, then $\sigma_s(\Split_d(\PP^n))$ is nondefective for all $n\geq 3$.
\end{corollary}

Fix an $s\in\NN$.  Note that, for large $n$, $\sigma_s(\Split_d(\PP^n))$ is nondefective for $d\geq 3$ by Proposition \ref{LargeN}.  For large $d$, $\sigma_s(\Split_d(\PP^n))$ is nondefective by Corollary \ref{S1Generalized}.  Therefore, there are only finitely many more cases to check.  (In fact, we end up not even needing Proposition \ref{LargeN}, for as soon as we check one subabundant case for some $d$, we get the rest of them by Theorem \ref{induction on n}.)

For example, consider $s=11$.  Refer to Figure \ref{SmallSExample} for visual reference.  We know that $\sigma_{11}(\Split_1(\PP^n))$, $\sigma_{11}(\Split_d(\PP^1))$, and $\sigma_{11}(\Split_d(\PP^2))$ are nondefective for all $n$ and $d$.  By Proposition \ref{DEqualsTwo}, $\sigma_{11}\Split_2(\PP^n)$ is nondefective for $n\leq 21$ and defective for $n\geq 22$.  By Theorem \ref{nEqualsThreeResult}, $\sigma_{11}(\Split_d(\PP^3))$ is nondefective for all $d\leq 10$, and by Theorem \ref{AboResult}, $\sigma_{11}(\Split_3(\PP^n))$ is nondefective for all $n\neq 11$.  By Corollary \ref{S1Generalized}, $\sigma_{11}(\Split_d(\PP^n))$ is nondefective for all $d\geq 12$ and $n\geq 3$.

From this point, we start to calculate dimensions of vector spaces in Macaulay2 to confirm the nondefectivity of the remaining cases, as shown in Proposition \ref{DimSecantVariety}.

First, we check $\sigma_{11}(\Split_3(\PP^{11}))$ to complete the cubic case.  Next, we check $\sigma_{11}(\Split_4(\PP^n)$ for $n\in\{4,\ldots,7\}$.  At this point, since $11<\frac{1}{4\cdot 7 + 1}\binom{7+4}{4}$, we may use Theorem \ref{induction on n} to complete the quartic case.  We continue in this fashion, incrementing $d$ until we reach $d=11$.  We check $\sigma_{11}(\Split_{11}(\PP^n))$ for $n\in\{3,4\}$, and since $11<\frac{1}{4\cdot 11+1}\binom{4+11}{11}$, the rest follows from Theorem \ref{induction on n}.  As a result of these calculations, we have proved Conjecture \ref{Conjecture} for $s=11$.

\begin{figure}[!ht]
\centering
\begin{tikzpicture}[scale=.5]
\fill[pattern=north west lines] (0,0) -- (15,0) -- (15,-25) -- (11,-25) -- (11,-2) -- (10,-2) -- (10,-3) -- (3,-3) -- (3,-10) -- (2,-10) -- (2,-11) -- (3,-11) -- (3,-25) -- (2,-25) -- (2,-21) -- (1,-21) -- (1,-25) -- (0,-25) -- (0,0);
\fill[pattern=crosshatch] (1,-21) -- (2,-21) -- (2,-25) -- (1,-25) -- (1,-21);
\fill[pattern=crosshatch dots] (3,-10) -- (2,-10) -- (2,-11) -- (3,-11) -- (3,-10);
\fill[pattern=crosshatch dots] (3,-3) -- (3,-7) -- (4,-7) -- (4,-6) -- (5,-6) -- (5,-5) -- (6,-5) -- (6,-4) -- (11,-4) -- (11,-2) -- (10,-2) -- (10,-3) -- (3,-3);
\draw[->] (0,0) -- (16,0) node[right] {$d$};
\draw[->] (0,0) -- (0,-26) node[below] {$n$};
\foreach \n in {1,...,25} \draw (0,-\n+.5) node[left] {$\n$};
\foreach \d in {1,...,15} \draw (\d-.5,0) node[above] {$\d$};
\foreach \n in {1,...,25} \draw (0,-\n) -- (15,-\n);
\foreach \d in {1,...,15} \draw (\d,0) -- (\d,-25);
\fill[pattern=north west lines] (18,-5) -- (19,-5) -- (19,-6) -- (18,-6);
\draw (18,-5) -- (19,-5) -- (19,-6) -- (18,-6) -- (18,-5);
\draw (19,-5.5) node[right] {known nondefective cases};
\fill[pattern=crosshatch] (18,-7) -- (19,-7) -- (19,-8) -- (18,-8);
\draw  (18,-7) -- (19,-7) -- (19,-8) -- (18,-8) -- (18,-7);
\draw (19,-7.5) node[right] {known defective cases};
\fill[pattern=crosshatch dots] (18,-9) -- (19,-9) -- (19,-10) -- (18,-10);
\draw (18,-9) -- (19,-9) -- (19,-10) -- (18,-10) -- (18,-9);
\draw (19,-9.5) node[right] {check in Macaulay2};
\draw (18,-11) -- (19,-11) -- (19,-12) -- (18,-12) -- (18,-11);
\draw (19,-11.5) node[right] {follows by Theorem \ref{induction on n}};
\end{tikzpicture}
\caption{Application of Proposition \ref{SmallSAlgorithm} for $s=11$\label{SmallSExample}}
\end{figure}
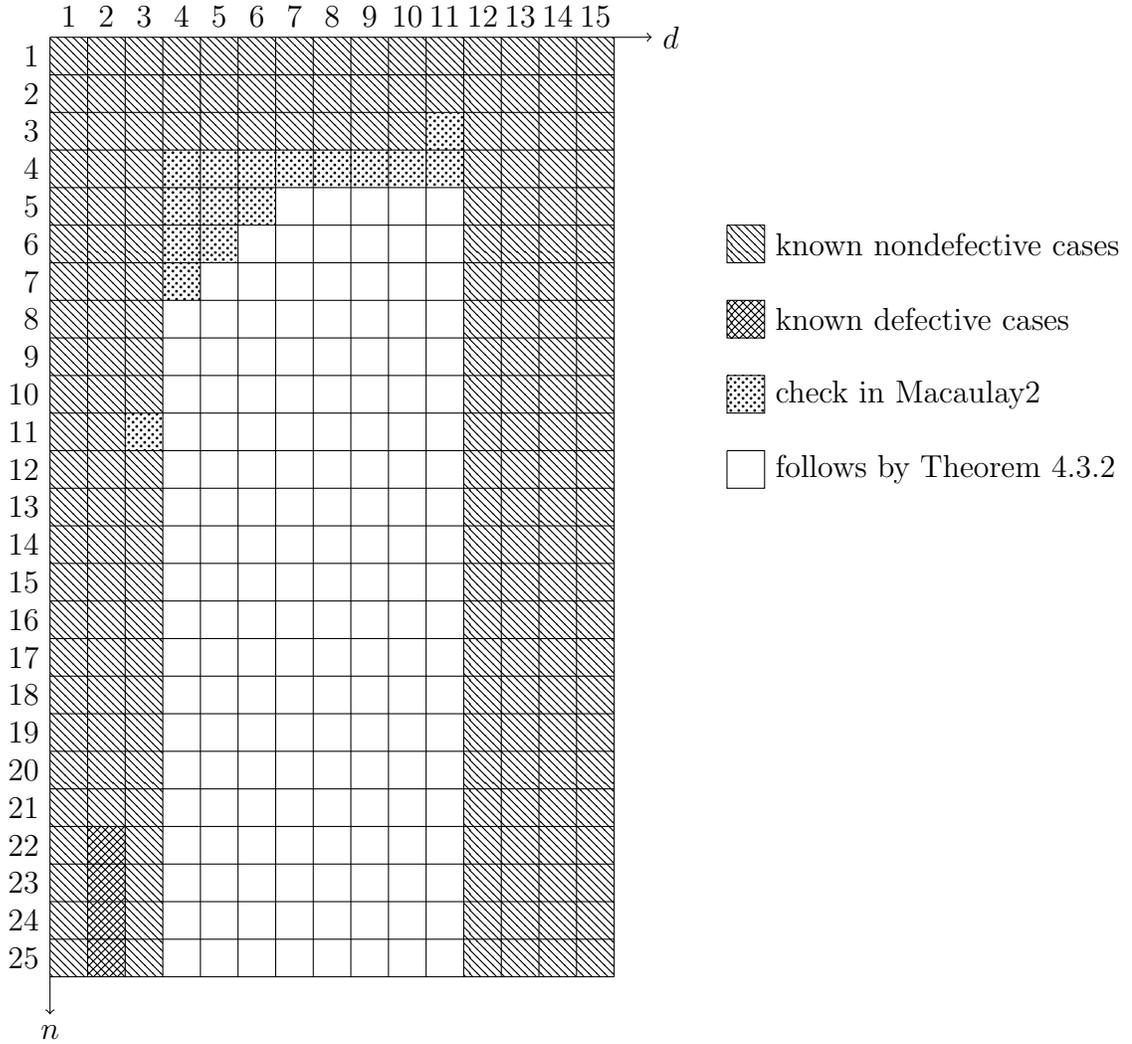

We make this argument precise in the following proposition.

Recall the following functions from Theorems \ref{AboResult} and \ref{nEqualsThreeResult}.
\begin{align*}
s_2(d) &=
\begin{cases}
\frac{1}{18}d^2+\frac{1}{3}d+1 &\text{ if }d\equiv 0\pmod{6}\\
\frac{1}{18}d^2+\frac{1}{3}d+\frac{1}{2} &\text{ if }d\equiv 3\pmod{6}\\
\frac{1}{18}d^2+\frac{7}{18}d+\frac{5}{9}&\text{ if }d\equiv 1,4,7\pmod{9}\\
\frac{1}{18}d^2+\frac{7}{18}d+1&\text{ if }d\equiv 2\pmod{9}\\
\frac{1}{18}d^2+\frac{7}{18}d+\frac{2}{3}&\text{ if }d\equiv 5\pmod{9}\\
\frac{1}{18}d^2+\frac{7}{18}d+\frac{1}{3}&\text{ if }d\equiv 8\pmod{9}\\
\end{cases}\\
s'_2(n) &=
\begin{cases}
\frac{1}{18}n^2+\frac{1}{3}n+1&\text{ if }n\equiv 0\pmod{6}\\
\frac{1}{18}n^2+\frac{7}{18}n+\frac{14}{9}&\text{ if }n\equiv 1\pmod{6}\\
\frac{1}{18}n^2+\frac{4}{9}n+\frac{8}{9}&\text{ if }n\equiv 2\pmod{6}\\
\frac{1}{18}n^2+\frac{1}{3}n+\frac{1}{2}&\text{ if }n\equiv 3\pmod{6}\\
\frac{1}{18}n^2+\frac{7}{18}n+\frac{5}{9}&\text{ if }n\equiv 4\pmod{6}\\
\frac{1}{18}n^2+\frac{4}{9}n+\frac{7}{18}&\text{ if }n\equiv 5\pmod{6}\\\end{cases}
\end{align*}

\begin{proposition}\label{SmallSAlgorithm}
Fix an $s\in\NN$.  If $\sigma_s(\Split_d(\PP^n))$ is nondefective for all of the following cases
\begin{enumerate}[(i)]
\item $d=3$ and $\min\{n:s<s'_2(n)\}\leq n\leq\min\left\{n:s\leq\frac{1}{3n+1}\binom{n+3}{3}\right\}$,
\item $4\leq d\leq\max\{d:s\geq s_2(d)\}$ and $4\leq n\leq\min\left\{n:s\leq\frac{1}{dn+1}\binom{n+d}{d}\right\}$, and
\item $\min\{d:s<s_2(d)\}\leq d\leq\max\{d:s>s_1(d)\}$ and \\$3\leq n\leq\min\left\{n:s\leq\frac{1}{dn+1}\binom{n+d}{d}\right\}$,
\end{enumerate}
then $\sigma_s(\Split_d(\PP^n))$ is nondefective for all $n,d\in\NN$ unless $d=2$ and $n\geq 2s$.
\end{proposition}

\begin{proof}

The $n=1$ case is trivial and the $n=2$ case follows from Theorem \ref{AboResultNEqualsTwo}, so we may assume $n\geq 3$.

The linear case is trivial and the quadratic case follows from Proposition \ref{DEqualsTwo}.  For all $d$ such that $s\leq s_1(d)$ and $n\geq 3$, $\sigma_s(\Split_d(\PP^n))$ is nondefective by Corollary~\ref{S1Generalized}.  Consequently, we may assume $3\leq d\leq\max\{d:s>s_1(d)\}$.

Fix one such $d$.  By assumption, $\sigma_s(\Split_d(\PP^n))$ is nondefective for the smallest $n$ such that $s\leq\frac{1}{dn+1}\binom{n+d}{d}$.  By Theorem \ref{induction on n}, it is also true for all larger $n$.

It remains to check the superabundant cases.  If $d=3$, then $\sigma_s(\Split_d(\PP^n))$ is nondefective if $s\geq s'_2(n)$ by Theorem \ref{AboResult}.  If $s\geq s_2(d)$, then $\sigma_s(\Split_d(\PP^3))$ is nondefective by Theorem \ref{nEqualsThreeResult}.

The remaining cases are all true by assumption.
\end{proof}

\begin{thmn}[\ref{SUpperBound}]
If $s\leq 30$, then $\sigma_s(\Split_d(\PP^n))$ is nondefective for all $n,d\in\NN$ unless $d=2$ and $2\leq s\leq\frac{n}{2}$.
\end{thmn}

\begin{proof}
By Proposition \ref{SmallSAlgorithm}, to prove that $\sigma_s(\Split_d(\PP^n))$ is nondefective for all $n,d\in\NN$ (except for the known defective cases), we need only check that finitely many cases are nondefective, for example by using Proposition \ref{DimSecantVariety}.  This is done for $3\leq s\leq 30$ in Macaulay2 Computation \ref{SmallSM2}.  (The $s=1$ case is trivial and the $s=2$ case is known.  See the discussion at the beginning of this section.)
\end{proof}

\pagebreak

\section{Comparison of results}

We have seen three different upper bounds for $s$ such that $\sigma_s(\Split_d(\PP^n))$ is nondefective for $n\geq 4$ and $d\geq 3$:
\begin{enumerate}[(i)]
\item $s\leq\frac{1}{3}n+1$ (Proposition \ref{LargeN}),
\item $s\leq 2^{n-3}c(n,d)$ (Theorem \ref{ExponentialBound}), and
\item $s\leq s_1(d)$ (Theorem \ref{S1AndS2}).
\end{enumerate}

Note that, depending on the values of $n$ and $d$, any one of these bounds may be the best.  Informally, if $n\gg d$, then (i) is the best bound.  If $d\gg n$, then (ii) is the best bound.  Otherwise, (iii) is the best bound.  We illustrate this in Figure \ref{BoundComparison}.

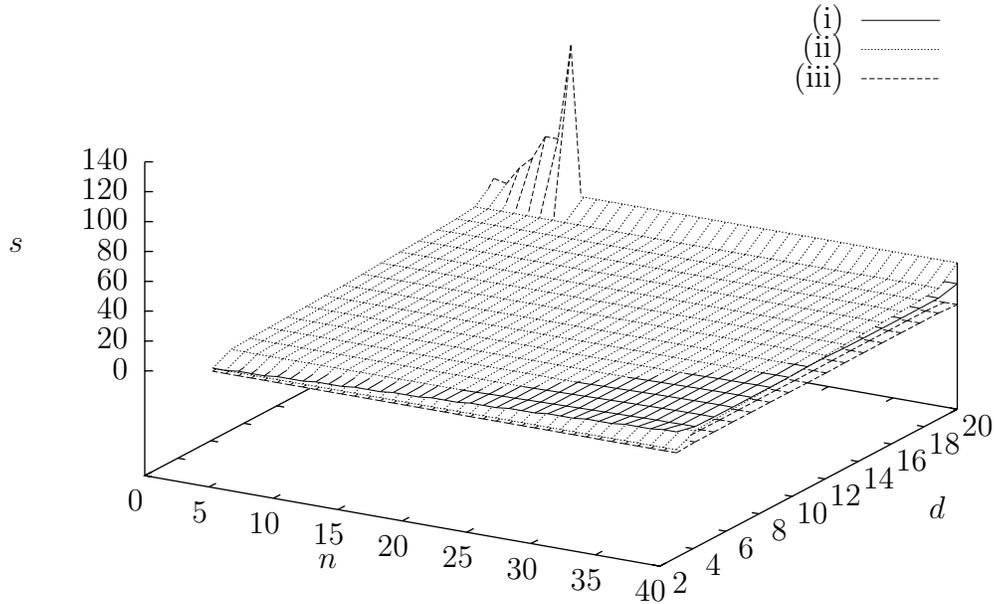
\begin{figure}[!ht]
\centering
\input{comparison.tex}
\caption{Comparison of upper bounds on $s$ for $n\geq 4$ and $d\geq 3$\label{BoundComparison}}
\end{figure}

\pagebreak

\chapter{Future work}

\section{Complete the $n=3$ and $d=3$ cases}

Using Theorems \ref{RestrictionInductionFixedDimension} and \ref{RestrictionInductionFixedDegree}, it is theoretically possible to prove Conjecture \ref{Conjecture} for any fixed $n$ or $d$.  Indeed, if the function $s$ is chosen such that $s(d)=\left\lfloor\frac{\binom{n+d}{d}}{dn+1}\right\rfloor$ (in the subabundant case) or $s(d)=\left\lceil\frac{\binom{n+d}{d}}{dn+1}\right\rceil$ (in the superabundant case) for all $d$, then the result will follow (or, in the case of Theorem \ref{RestrictionInductionFixedDegree}, these identities hold for $s(n)$ for all $n$).  However, in practice, for large $n$ or $d$, we will need a large step size $\ell$ in order to define such a function.  Therefore, some of the base cases needed will require calculating the ranks of very large matrices, which may not be possible with the available computing power.

In this section, we provide functions which could be used, along with the required base cases, to complete the proof of Conjecture \ref{Conjecture} for $n=3$ and $d=3$

Define the following functions.
\begin{align*}
\underline s(d) &= \frac{1}{18}d^2+\frac{17}{54}d+\frac{a(d)}{27}\\
\overline s(d) &= \underline s(d)+1,
\end{align*}
where the values of $a(d)$ are given in Table \ref{ValuesOfA} for all $d=27t+r$, $t\in\ZZ$.

\begin{table}[h]
\centering
\begin{equation*}
\begin{array}{|r|r||r|r||r|r|}
\hline
r & a(27t+r) & r & a(27t+r) & r & a(27t+r) \\\hline
0 & 0   & 9  & -9  & 18 & -9 \\
1 & -10 & 10 & 8   & 19 & -1 \\
2 & 4   & 11 & -5  & 20 & 13 \\
3 & -12 & 12 & 6   & 21 & -3 \\
4 & -4  & 13 & -13 & 22 & 5 \\
5 & 1   & 14 & -8  & 23 & 10 \\
6 & 3   & 15 & -6  & 24 & 12 \\
7 & 2   & 16 & -7  & 25 & 11 \\
8 & -2  & 17 & -11 & 26 & 7 \\
\hline
\end{array}
\end{equation*}
\caption{Values of $a(d)$\label{ValuesOfA}}
\end{table}

\begin{lemma}\label{TheseAreTheCorrectFunctions}
If $d\in\NN\setminus\{1,3,13\}$, then $\underline s(d)=\left\lfloor\frac{\binom{d+3}{3}}{3d+1}\right\rfloor$.  If $d\in\NN$, then $\overline s(d)=\left\lceil\frac{\binom{d+3}{3}}{3d+1}\right\rceil$.
\end{lemma}

\begin{proof}
First, note that
\begin{equation*}
\frac{\binom{d+3}{3}}{3d+1} = \frac{1}{18}d^2+\frac{17}{54}d+\frac{41}{81}+\frac{40/81}{3d+1}.\\
\end{equation*}

Suppose $d=27t+r$, where $r\in\{0,\ldots,26\}$.  Then we have
\begin{align*}
\frac{\binom{d+3}{3}}{3d+1} &= \frac{1}{18}(27t+r)^2+\frac{17}{54}(27t+r)+\frac{41}{81}+\frac{40/81}{3d+1}\\
&= \frac{81}{2}t^2+\left(3r+\frac{17}{2}\right)t+\frac{1}{18}r^2+\frac{17}{54}r+\frac{41}{81}+\frac{40/81}{3d+1}.
\end{align*}

Note that $\frac{81}{2}t^2+\left(3r+\frac{17}{2}\right)t\in\ZZ$ for all $t\in\ZZ$.  Let $f(d)=\frac{1}{18}d^2+\frac{17}{54}$.  Therefore, if $g(r)=\frac{1}{18}r^2+\frac{17}{54}r+\frac{41}{81}$, then
\begin{align*}
\left\lfloor\frac{\binom{d+3}{3}}{3d+1}\right\rfloor &= \frac{81}{2}t^2+\left(3r+\frac{17}{2}\right)t + \left\lfloor g(r)+\frac{40/81}{3d+1}\right\rfloor\\
&= f(d)-f(r)+\left\lfloor g(r)+\frac{40/81}{3d+1}\right\rfloor\\
\left\lceil\frac{\binom{d+3}{3}}{3d+1}\right\rceil &= f(d)-f(r)+\left\lceil g(r)+\frac{40/81}{3d+1}\right\rceil.
\end{align*}
Note that, for sufficiently large $d$, the $\frac{40/81}{3d+1}$ term will have no effect on the floor and ceiling functions, provided that $g(r)\not\in\ZZ$.  In this case, let $h(r)=\max\{d:\left\lfloor g(r)+\frac{40/81}{3d+1}\right\rfloor>\left\lfloor g(r)\right\rfloor\}$.  We see then that, if $d>h(r)$,
\begin{align*}
\left\lfloor\frac{\binom{d+3}{3}}{3d+1}\right\rfloor &= f(d)-f(r)+\left\lfloor g(r)\right\rfloor\\
\intertext{and}
\left\lceil\frac{\binom{d+3}{3}}{3d+1}\right\rceil &= f(d)-f(r)+\left\lfloor g(r)\right\rfloor+1.
\end{align*}

It remains to check that $g(r)$ is not an integer for all $r$ and to find the values of $-f(r)+\left\lfloor g(r)\right\rfloor$ and $h(r)$ for all $r$.  These may all be calculated simply and are shown in Table \ref{ProofCalculations}.

This proves the result for all $d\not\in\{1,3,13\}$.  For these three cases, it is easy to check that $\frac{\binom{d+3}{3}}{3d+1}\in\ZZ$, and that we have $\overline s(d)=\frac{\binom{d+3}{3}}{3d+1}$.
\end{proof}

\begin{table}[h]
\centering
\begin{equation*}
\begin{array}{|c|c|c|c||c|c|c|c|}
\hline
r & g(r) & -f(r)+\left\lfloor g(r)\right\rfloor & h(r) & r & g(r) & -f(r)+\left\lfloor g(r)\right\rfloor & h(r)\\ \hline
0 & 41/81 & 0 & 0 & 14 & 1280/81 & -8/27 & 0 \\
1 & 71/81 & -10/27 & 1 &15 & 1436/81 & -2/9 & 0 \\
2 & 110/81 & 4/27 & 0 & 16 & 1601/81 & -7/27 & 0 \\
3 & 158/81 & -4/9 & 3 & 17 & 1775/81 & -11/27 & 1 \\
4 & 215/81 & -4/27 & 0 & 18 & 1958/81 & 1/3 & 0 \\
5 & 281/81 & 1/27 & 0 & 19 & 2150/81 & -1/27 & 0 \\
6 & 356/81 & 1/9 & 0 & 20 & 2351/81 & 13/27 & 0 \\
7 & 440/81 & 2/27 & 0 & 21 & 2561/81 & -1/9 & 0 \\
8 & 533/81 & -2/27 & 0 & 22 & 2780/81 & 5/27 & 0 \\
9 & 635/81 & -1/3 & 0 & 23 & 3008/81 & 10/27 & 0 \\
10 & 746/81 & 8/27 & 0 & 24 & 3245/81 & 4/9 & 0 \\
11 & 866/81 & -5/27 & 0 & 25 & 3491/81 & 11/27 & 0 \\
12 & 995/81 & 2/9 & 0 & 26 & 3746/81 & 7/27 & 0 \\
13 & 1133/81 & -13/27 & 13 & & & &\\
\hline
\end{array}
\end{equation*}
\caption{Calculations for the proof of Lemma \ref{TheseAreTheCorrectFunctions}\label{ProofCalculations}}
\end{table}

Next, note that 
\begin{equation*}
3!\left(3\cdot\frac{1}{18}\right)=1
\end{equation*}
and so, by Lemma \ref{AnEquiabundant}, $\statement A_3(3,d,27,\underline s,0,0,0)$ and $\statement A_3(3,d,27,\overline s,0,0,0)$ are equiabundant for all $d\geq 82$, as are $\statement B_3(n,3,27,\underline s,0,0,0)$ and $\statement B_3(n,3,27,\overline s,0,0,0)$ for all $n\geq 82$.

It remains to check the base cases in order to apply Theorems \ref{RestrictionInductionFixedDimension} and \ref{RestrictionInductionFixedDegree}.  However, the matrices required to do this are very large.  For example, in order to check the truth of $\statement A_3(3,82,27,\underline s,0,0,0)$, one would need to calculate the rank of a 98,770$\times$98,770 matrix, which is beyond the computing power available to the author at the time of this writing.

\pagebreak

\section{Secant varieties of $\Seg(\PP^m\times\Split_d(\PP^n))$}

The \define{Segre map} $\Seg$ embeds the biprojective space $\PP^m\times\PP^n$ into $\PP^{(m+1)(n+1)-1}$ by $([u],[v])\mapsto[u\otimes v]$.  The image $\Seg(\PP^m\times\PP^n)$ is known as a \define{Segre variety}, and the dimensions of its secant varieties are well known, as $\sigma_s(\Seg(\PP^m\times\PP^n))$ corresponds to all of the $(m+1)\times(n+1)$ matrices of rank at most $s$.  (The more general question of the dimensions of $\sigma_s(\Seg(\PP^{n_1}\times\cdots\times\PP^{n_k}))$ has been the source of much active research.  See for example \cite{AOP,AladpooshHaghighi,CGG3,CGG2,CGG4}.)

One may study the secant varieties of $\Seg(\PP^m\times\Split_d(\PP^n))$ to attempt to solve the following variation of Waring's problem.

\begin{problem}
What is the smallest $s$ such that any $m+1$ generic polynomials in $n+1$ variables of degree $d$ may be written as linear combinations of at most $s$ of the same completely decomposable forms?
\end{problem}

As with Problems \ref{WaringPolynomials} and \ref{WaringCDF}, the strategy is to determine the smallest $s$ such that the secant variety $\sigma_s(\Seg(\PP^m\times\Split_d(\PP^n)))$ fills the ambient space.  We have
\begin{equation*}
\expdim\sigma_s(\Seg(\PP^m\times\Split_d(\PP^n)))=\min\left\{s(m+dn+1),(m+1)\binom{n+d}{d}\right\}-1
\end{equation*}

Therefore, provided that $\sigma_s(\Seg(\PP^m\times\Split_d(\PP^n)))$ is nondefective, it will fill the ambient space when
\begin{equation*}
s\geq\left\lceil\frac{(m+1)\binom{n+d}{d}}{m+dn+1}\right\rceil.
\end{equation*}

It remains to determine which cases are defective.  The following family is known to be defective (see \cite[Proposition 3.9]{BuczynskiLandsberg} or \cite[Theorem 3.4]{AboWan}).

\begin{proposition}\label{UnbalancedSegPmSplit}
Suppose $m>\binom{n+d}{d}-dn$.  Then $\sigma_s(\Seg(\PP^m\times\Split_d(\PP^n)))$ is defective if and only if $\binom{n+d}{d}-dn<s\leq\min\left\{m,\binom{n+d}{d}-1\right\}$.
\end{proposition}

Let $R$ be as usual and $S=\kk[y_0,\ldots,y_m]$.  Then, as in Proposition \ref{DimSecantVariety}, we may find the dimension of the secant variety by calculating the dimension of a vector space.

\begin{proposition}
For any generic $g_1,\ldots,g_s\in S_1$ and $\tuple f_1,\ldots,\tuple f_s\in R_1^{d}$,
\begin{equation*}
\dim\sigma_s(\Seg(\PP^m\times\Split_d(\PP^n)))=\dim\sum_{i=1}^s\left[S_1\otimes\prod\tuple f_i+\sum_{j=1}^dg_i\otimes\pi_j(\tuple f_i)R_1\right]-1.
\end{equation*}
\end{proposition}

Based on a number of Macaulay2 computations, the author believes that it is reasonable to make the following conjecture.

\begin{conjecture}
If $m\geq 1$, then $\sigma_s(\Seg(\PP^m\times\Split_d(\PP^n)))$ is nondefective for all cases except those outlined in Proposition \ref{UnbalancedSegPmSplit}
\end{conjecture}

It would be relatively straightforward to adapt the methods of this dissertation for $\sigma_s(\Split_d(\PP^n))$ to $\sigma_s(\Seg(\PP^m\times\Split_d(\PP^n)))$ in order to provide a partial proof of this conjecture.
\pagebreak

\addcontentsline{toc}{chapter}{References}
\renewcommand\bibname{References}

\bibliography{sources}{}
\bibliographystyle{abbrv}

\begin{appendices}
\chapter{Macaulay2 Code}
\label{M2Computations}

\section{Basic functions}

In this section, we outline the Macaulay2 functions used to check the truth and abundancy of the statement $\statement A_i(n,d,\ell,s,t,u,v)$, as outlined in Sections \ref{Definitions} and \ref{FixedDegree}

First, we implement the backward difference operator $\nabla$.

\lstset{columns=flexible,breaklines=true,basicstyle=\ttfamily}
\begin{lstlisting}
bd = (f,i,l) -> x -> sum(i+1,j->(-1)^j*binomial(i,j)*f(x-j*l))
\end{lstlisting}

Next, we have a function which chooses generic points in $V^{d}$ for some subspace $V\subset R_1$.  Note that we implement this by choosing random elements of an ideal $I$.

\begin{lstlisting}
genericLinearForms = (I,d) -> (
     R := ring I;
     apply(d,i->first flatten entries (gens I*random(R^(numgens I),R^1)))
     )
\end{lstlisting}

We then implement the map $\pi_j$.  (See Section \ref{TerraciniSection}.)

\begin{lstlisting}
p = (j,f) -> product remove(f,j-1)
\end{lstlisting}

By definition, $s$, $t$, $u$, and $v$ are functions.  However, in practice, in cases when $i=0$ (\textit{e.g.}, in Section \ref{SplittingInduction}), we need only their values at $d$.  We need a way to correctly interpret them as functions when they are given as numbers.

\begin{lstlisting}
convertToFunctions = (s,t,u,v) -> (
     if instance(s,Number) then (
	  s' := s;
	  s = x -> s';
	  );
     if instance(t,Number) then (
	  t' := t;
	  t = x -> t';
	  );
     if instance(u,Number) then (
	  u' := u;
	  u = x -> u';
	  );
     if instance(v,Number) then (
	  v' := v;
	  v = x -> v';
	  );
     (s,t,u,v)
     )
\end{lstlisting}

We then build the vector space $A_i(n,d,\ell,s,t,u,v)$ itself.  Note that it is implemented as an ideal.  As explained in Section \ref{ComputationalTechniques}, we work over a finite field.  Also note that it is not given exactly as defined, but rather an optimized version is given as suggested by the proof of Lemma \ref{AiExpdim}

\begin{lstlisting}
A = (i,n,d,l,s,t,u,v) -> (
     (s,t,u,v) = convertToFunctions(s,t,u,v);
     R := ZZ/32003[x_0..x_n];
     M := ideal vars R;
     scan(toList(1..i),j->g_j=genericLinearForms(M,l));
     scan(toList(1..i),j->
	  scan(toList(1..(bd(s,i-1,l))(d-l)),k->
	       f_(j,k)=genericLinearForms(M,d-l)
	       )
	  );
     scan(toList(1..i),j->
	  scan(toList(1..(bd(u,i-1,l))(d-l)),k->
	       f''_(j,k)=genericLinearForms(M,d-l+1)
	       )
	  );
     scan(toList(1..(bd(s,i,l))(d)),j->
	  f_j = genericLinearForms(M,d)
	  );
     scan(toList(1..(bd(t,i,l))(d)),j->
	  f'_j = genericLinearForms(M,d+1)
	  );
     scan(toList(1..(bd(u,i,l))(d)),j->
	  f''_j = genericLinearForms(M,d+1)
	  );
     scan(toList(1..(bd(v,i,l))(d)),j->
	  f'''_j = genericLinearForms(M,d)
	  );
     A := ideal 0_R
     	  + sum(toList(1..i),j->ideal (product g_j*basis(d-l,R)))
     	  + sum(toList(1..i),j->
	       sum(toList(1..(bd(s,i-1,l))(d-l)),k->
	       	    sum(toList((d-l+1)..d),m->ideal (p(m,f_(j,k)|g_j)*genericLinearForms(M,n)))))
	  + sum(toList(1..i),j->
	       sum(toList(1..(bd(u,i-1,l))(d-l)),k->
	       	    sum((d-l+2)..(d+1),m->ideal p(m,f''_(j,k)|g_j))))
     	  + sum(toList(1..(bd(s,i,l))(d)),j->
	       sum(toList(1..d),m->ideal (p(m,f_j)*vars R)))
     	  + sum(toList(1..(bd(t,i,l))(d)),j->ideal p(1,f'_j))
     	  + sum(toList(1..(bd(u,i,l))(d)),j->
	       sum(toList(1..(d+1)),m->ideal p(m,f''_j)))
     	  + sum(toList(1..(bd(v,i,l))(d)),j->ideal (p(1,f'''_j)*vars R));
     trim A
     )
\end{lstlisting}

Finally, the function $a_i(n,d,\ell,s,t,u,v)$ is given, as is the statement $\statement A_i(n,d,\ell,s,t,\allowbreak u,v)$ and its abundancy.

\begin{lstlisting}
a = (i,n,d,l,s,t,u,v) -> (
     (s,t,u,v) = convertToFunctions(s,t,u,v);
     sum(toList(1..i),j->(-1)^(j-1)*binomial(i,j)*binomial(n+d-j*l,d-j*l))
     	  + i*l*n*(bd(s,i-1,l))(d-l)+i*l*(bd(u,i-1,l))(d-l)
	  + (d*n+1)*(bd(s,i,l))(d) + (bd(t,i,l))(d)
	  + (d+1)*(bd(u,i,l))(d) + (n+1)*(bd(v,i,l))(d)
     )

statementA = (i,n,d,l,s,t,u,v) -> 
     numgens A(i,n,d,l,s,t,u,v) ==
     	  min{a(i,n,d,l,s,t,u,v),binomial(n+d,d)}

subabundantA = (i,n,d,l,s,t,u,v) ->
     a(i,n,d,l,s,t,u,v) <= binomial(n+d,d)

superabundantA = (i,n,d,l,s,t,u,v) ->
     a(i,n,d,l,s,t,u,v) >= binomial(n+d,d)
\end{lstlisting}

\pagebreak
\section{Restriction Induction}

The following function checks the base cases for Theorem \ref{RestrictionInductionFixedDimension}.

\begin{lstlisting}
restrictionInductionFixedDimension = (n,l,s,t,u,v) -> (
     (all(toList(1..(l*n+1)),d->subabundantA(floor((d-1)/l),n,d,l,s,t,u,v))     
	  or all(toList(1..(l*n+1)),d->superabundantA(floor((d-1)/l),n,d,l,s,t,u,v)))
     	  and all(toList(1..(l*n+1)),d->(
		    print("checking d = "|toString d);
		    time statementA(floor((d-1)/l),n,d,l,s,t,u,v))
	       )
     )
\end{lstlisting}

\begin{computation}\label{NEqualsThreeM2}
We check the base cases for Theorem \ref{nEqualsThreeResult}.

\begin{lstlisting}
i18 : s1 = d -> (
     r := d % 9;
     sub(1/18*d^2 + 5/18*d + {0,2/3,2/9,2/3,0,2/9,1/3,1/3,2/9}_r,ZZ)
     )
                  
o18 = s1

o18 : FunctionClosure

i19 : restrictionInductionFixedDimension(3,9,s1,0,0,0)
checking d = 1
     -- used 0.0029162 seconds
checking d = 2
     -- used 0.00313631 seconds
checking d = 3
     -- used 0.00570147 seconds
checking d = 4
     -- used 0.00938839 seconds
checking d = 5
     -- used 0.0242013 seconds
checking d = 6
     -- used 0.0647426 seconds
checking d = 7
     -- used 0.158475 seconds
checking d = 8
     -- used 0.355998 seconds
checking d = 9
     -- used 0.750897 seconds
checking d = 10
     -- used 1.64098 seconds
checking d = 11
     -- used 3.00147 seconds
checking d = 12
     -- used 5.71705 seconds
checking d = 13
     -- used 11.5449 seconds
checking d = 14
     -- used 17.1712 seconds
checking d = 15
     -- used 30.4635 seconds
checking d = 16
     -- used 50.1195 seconds
checking d = 17
     -- used 80.062 seconds
checking d = 18
     -- used 122.41 seconds
checking d = 19
     -- used 127.514 seconds
checking d = 20
     -- used 185.234 seconds
checking d = 21
     -- used 254.379 seconds
checking d = 22
     -- used 354.713 seconds
checking d = 23
     -- used 477.995 seconds
checking d = 24
     -- used 640.095 seconds
checking d = 25
     -- used 821.756 seconds
checking d = 26
     -- used 1074.18 seconds
checking d = 27
     -- used 1423.75 seconds
checking d = 28
     -- used 780.115 seconds

o19 = true

i20 : s2'' = d -> (
     r := d % 9;
     sub(1/18*d^2 + 7/18*d + {0,5/9,1,1/3,5/9,2/3,2/3,5/9,1/3}_r,ZZ)
     )

o20 = s2''

o20 : FunctionClosure

i21 : restrictionInductionFixedDimension(3,9,s2'',0,0,0)
checking d = 1
     -- used 0.00279336 seconds
checking d = 2
     -- used 0.00361807 seconds
checking d = 3
     -- used 0.00534276 seconds
checking d = 4
     -- used 0.0124839 seconds
checking d = 5
     -- used 0.0326489 seconds
checking d = 6
     -- used 0.0816455 seconds
checking d = 7
     -- used 0.19598 seconds
checking d = 8
     -- used 0.430366 seconds
checking d = 9
     -- used 0.879879 seconds
checking d = 10
     -- used 1.84921 seconds
checking d = 11
     -- used 5.3462 seconds
checking d = 12
     -- used 6.34763 seconds
checking d = 13
     -- used 13.1411 seconds
checking d = 14
     -- used 21.5793 seconds
checking d = 15
     -- used 34.5114 seconds
checking d = 16
     -- used 59.0605 seconds
checking d = 17
     -- used 90.318 seconds
checking d = 18
     -- used 136.225 seconds
checking d = 19
     -- used 139.097 seconds
checking d = 20
     -- used 201.146 seconds
checking d = 21
     -- used 273.597 seconds
checking d = 22
     -- used 374.172 seconds
checking d = 23
     -- used 496.67 seconds
checking d = 24
     -- used 655.896 seconds
checking d = 25
     -- used 872.895 seconds
checking d = 26
     -- used 1112.72 seconds
checking d = 27
     -- used 1430.58 seconds
checking d = 28
     -- used 781.155 seconds

o21 = true
\end{lstlisting}
\end{computation}

\begin{computation}\label{TildeSM2}
We check the base cases for Theorem \ref{TildeSBound}.

\begin{lstlisting}
i11 : sTilde = d -> (
     r := d % 6;
     sub(1/24*d^2 + {1/12*d,1/6*d-5/24,1/12*d-1/3,1/6*d+1/8,1/12*d,1/6*d+1/8}_r,ZZ)
     )

o11 = sTilde

o11 : FunctionClosure

i12 : restrictionInductionFixedDimension(3,6,sTilde,sTilde,sTilde,sTilde)
checking d = 1
     -- used 0.0174115 seconds
checking d = 2
     -- used 0.0122132 seconds
checking d = 3
     -- used 0.0220609 seconds
checking d = 4
     -- used 0.0298028 seconds
checking d = 5
     -- used 0.0777791 seconds
checking d = 6
     -- used 0.135792 seconds
checking d = 7
     -- used 0.456312 seconds
checking d = 8
     -- used 0.759432 seconds
checking d = 9
     -- used 1.9626 seconds
checking d = 10
     -- used 3.47087 seconds
checking d = 11
     -- used 7.7691 seconds
checking d = 12
     -- used 12.5589 seconds
checking d = 13
     -- used 18.2089 seconds
checking d = 14
     -- used 26.8096 seconds
checking d = 15
     -- used 43.7815 seconds
checking d = 16
     -- used 61.4002 seconds
checking d = 17
     -- used 91.766 seconds
checking d = 18
     -- used 125.668 seconds
checking d = 19
     -- used 83.8318 seconds

o12 = true
\end{lstlisting}
\end{computation}

\pagebreak
\section{Small $s$}

The following functions check the bases cases for a given $s$ from Theorem \ref{SmallSAlgorithm}.
\begin{lstlisting}
s1 = d -> (
     r := d % 9;
     sub(1/18*d^2 + 5/18*d + {0,2/3,2/9,2/3,0,2/9,1/3,1/3,2/9}_r,ZZ)
     )
s2 = d -> (
     if d % 3 == 0 then (
	  r := sub((d % 6)/3,ZZ);
	  sub(1/18*d^2+1/3*d + {1,1/2}_r,ZZ)
	  )
     else (
	  r = d % 9;
	  sub(1/18*d^2+7/18*d+{0,5/9,1,0,5/9,2/3,0,5/9,1/3}_r,ZZ)
	  )
     )
s2' = d -> (
     r := d % 6;
     sub(1/18*d^2+{1/3*d+1,7/18*d+14/9,4/9*d+8/9,1/3*d+1/2,7/8*d+5/9,4/9*d+7/18}_r,ZZ)
     )
smallS = s -> (
     ndPairs = {};
     d := 3;
     while s > s1(d) do (
	  if d == 3 then (
	       n := 1;
	       while s > 1/(3*(n-1)+1)*binomial(n+2,3) do (
		    if s < s2'(n) then ndPairs = append(ndPairs,(n,d));
		    n = n + 1
		    )
	       )
	  else if s >= s2(d) then (
	       n = 4;
	       while s > 1/(d*(n-1)+1)*binomial(n+d-1,d) do (
		    ndPairs = append(ndPairs,(n,d));
		    n = n + 1
		    )
	       )
	  else (
	       n = 3;
	       while s > 1/(d*(n-1)+1)*binomial(n+d-1,d) do (
		    ndPairs = append(ndPairs,(n,d));
		    n = n + 1
		    )
	       );    
	  d = d + 1
	  );
     	  all(ndPairs,(n,d)->statementA(0,n,d,0,s,0,0,0))
     )
\end{lstlisting}

\begin{computation}\label{SmallSM2}
In this computation, we run \texttt{smallS} for each $s$ starting at $s=3$ and continuing until the computer runs out of memory.  This forms the proof of Theorem \ref{SUpperBound}.

\begin{lstlisting}
i15 : s = 3
o15 = 3
i16 : while true do (
           print("checking s = "|toString s);
           print("     "|toString time smallS s);
           s = s + 1
           )
checking s = 3
     -- used 0.053642 seconds
     true
checking s = 4
     -- used 0.181827 seconds
     true
checking s = 5
     -- used 0.621823 seconds
     true
checking s = 6
     -- used 1.7983 seconds
     true
checking s = 7
     -- used 5.56529 seconds
     true
checking s = 8
     -- used 12.7186 seconds
     true
checking s = 9
     -- used 18.3932 seconds
     true
checking s = 10
     -- used 40.9189 seconds
     true
checking s = 11
     -- used 83.181 seconds
     true
checking s = 12
     -- used 102.282 seconds
     true
checking s = 13
     -- used 194.171 seconds
     true
checking s = 14
     -- used 230.404 seconds
     true
checking s = 15
     -- used 453.733 seconds
     true
checking s = 16
     -- used 897.747 seconds
     true
checking s = 17
     -- used 962.889 seconds
     true
checking s = 18
     -- used 1664.02 seconds
     true
checking s = 19
     -- used 1855.09 seconds
     true
checking s = 20
     -- used 3448.47 seconds
     true
checking s = 21
     -- used 3700.53 seconds
     true
checking s = 22
     -- used 6155.73 seconds
     true
checking s = 23
     -- used 6798.77 seconds
     true
checking s = 24
     -- used 7599.76 seconds
     true
checking s = 25
     -- used 12280.4 seconds
     true
checking s = 26
     -- used 13260.5 seconds
     true
checking s = 27
     -- used 20752.3 seconds
     true
checking s = 28
     -- used 22647.2 seconds
     true
checking s = 29
     -- used 23984.9 seconds
     true
checking s = 30
     -- used 37101.7 seconds
     true
\end{lstlisting}
\end{computation}

\end{appendices}

\end{document}

%% file: comparison.tex
\begingroup
  \makeatletter
  \providecommand\color[2][]{%
    \GenericError{(gnuplot) \space\space\space\@spaces}{%
      Package color not loaded in conjunction with
      terminal option `colourtext'%
    }{See the gnuplot documentation for explanation.%
    }{Either use 'blacktext' in gnuplot or load the package
      color.sty in LaTeX.}%
    \renewcommand\color[2][]{}%
  }%
  \providecommand\includegraphics[2][]{%
    \GenericError{(gnuplot) \space\space\space\@spaces}{%
      Package graphicx or graphics not loaded%
    }{See the gnuplot documentation for explanation.%
    }{The gnuplot epslatex terminal needs graphicx.sty or graphics.sty.}%
    \renewcommand\includegraphics[2][]{}%
  }%
  \providecommand\rotatebox[2]{#2}%
  \@ifundefined{ifGPcolor}{%
    \newif\ifGPcolor
    \GPcolorfalse
  }{}%
  \@ifundefined{ifGPblacktext}{%
    \newif\ifGPblacktext
    \GPblacktexttrue
  }{}%
  \let\gplgaddtomacro\g@addto@macro
  \gdef\gplbacktext{}%
  \gdef\gplfronttext{}%
  \makeatother
  \ifGPblacktext
    \def\colorrgb#1{}%
    \def\colorgray#1{}%
  \else
    \ifGPcolor
      \def\colorrgb#1{\color[rgb]{#1}}%
      \def\colorgray#1{\color[gray]{#1}}%
      \expandafter\def\csname LTw\endcsname{\color{white}}%
      \expandafter\def\csname LTb\endcsname{\color{black}}%
      \expandafter\def\csname LTa\endcsname{\color{black}}%
      \expandafter\def\csname LT0\endcsname{\color[rgb]{1,0,0}}%
      \expandafter\def\csname LT1\endcsname{\color[rgb]{0,1,0}}%
      \expandafter\def\csname LT2\endcsname{\color[rgb]{0,0,1}}%
      \expandafter\def\csname LT3\endcsname{\color[rgb]{1,0,1}}%
      \expandafter\def\csname LT4\endcsname{\color[rgb]{0,1,1}}%
      \expandafter\def\csname LT5\endcsname{\color[rgb]{1,1,0}}%
      \expandafter\def\csname LT6\endcsname{\color[rgb]{0,0,0}}%
      \expandafter\def\csname LT7\endcsname{\color[rgb]{1,0.3,0}}%
      \expandafter\def\csname LT8\endcsname{\color[rgb]{0.5,0.5,0.5}}%
    \else
      \def\colorrgb#1{\color{black}}%
      \def\colorgray#1{\color[gray]{#1}}%
      \expandafter\def\csname LTw\endcsname{\color{white}}%
      \expandafter\def\csname LTb\endcsname{\color{black}}%
      \expandafter\def\csname LTa\endcsname{\color{black}}%
      \expandafter\def\csname LT0\endcsname{\color{black}}%
      \expandafter\def\csname LT1\endcsname{\color{black}}%
      \expandafter\def\csname LT2\endcsname{\color{black}}%
      \expandafter\def\csname LT3\endcsname{\color{black}}%
      \expandafter\def\csname LT4\endcsname{\color{black}}%
      \expandafter\def\csname LT5\endcsname{\color{black}}%
      \expandafter\def\csname LT6\endcsname{\color{black}}%
      \expandafter\def\csname LT7\endcsname{\color{black}}%
      \expandafter\def\csname LT8\endcsname{\color{black}}%
    \fi
  \fi
  \setlength{\unitlength}{0.0500bp}%
  \begin{picture}(8502.00,5668.00)%
    \gplgaddtomacro\gplbacktext{%
    }%
    \gplgaddtomacro\gplfronttext{%
      \csname LTb\endcsname%
      \put(6458,4944){\makebox(0,0)[r]{\strut{}(i)}}%
      \csname LTb\endcsname%
      \put(6458,4724){\makebox(0,0)[r]{\strut{}(ii)}}%
      \csname LTb\endcsname%
      \put(6458,4504){\makebox(0,0)[r]{\strut{}(iii)}}%
      \csname LTb\endcsname%
      \put(1081,1306){\makebox(0,0){\strut{} 0}}%
      \put(1566,1221){\makebox(0,0){\strut{} 5}}%
      \put(2051,1135){\makebox(0,0){\strut{} 10}}%
      \put(2537,1050){\makebox(0,0){\strut{} 15}}%
      \put(3022,964){\makebox(0,0){\strut{} 20}}%
      \put(3508,879){\makebox(0,0){\strut{} 25}}%
      \put(3993,794){\makebox(0,0){\strut{} 30}}%
      \put(4477,708){\makebox(0,0){\strut{} 35}}%
      \put(4963,623){\makebox(0,0){\strut{} 40}}%
      \put(2570,873){\makebox(0,0){\strut{}$n$}}%
      \put(5185,684){\makebox(0,0){\strut{} 2}}%
      \put(5435,816){\makebox(0,0){\strut{} 4}}%
      \put(5684,948){\makebox(0,0){\strut{} 6}}%
      \put(5933,1079){\makebox(0,0){\strut{} 8}}%
      \put(6182,1211){\makebox(0,0){\strut{} 10}}%
      \put(6431,1342){\makebox(0,0){\strut{} 12}}%
      \put(6680,1474){\makebox(0,0){\strut{} 14}}%
      \put(6929,1605){\makebox(0,0){\strut{} 16}}%
      \put(7178,1737){\makebox(0,0){\strut{} 18}}%
      \put(7427,1868){\makebox(0,0){\strut{} 20}}%
      \put(7163,1248){\makebox(0,0){\strut{}$d$}}%
      \put(1063,2300){\makebox(0,0)[r]{\strut{} 0}}%
      \put(1063,2525){\makebox(0,0)[r]{\strut{} 20}}%
      \put(1063,2751){\makebox(0,0)[r]{\strut{} 40}}%
      \put(1063,2975){\makebox(0,0)[r]{\strut{} 60}}%
      \put(1063,3201){\makebox(0,0)[r]{\strut{} 80}}%
      \put(1063,3426){\makebox(0,0)[r]{\strut{} 100}}%
      \put(1063,3652){\makebox(0,0)[r]{\strut{} 120}}%
      \put(1063,3877){\makebox(0,0)[r]{\strut{} 140}}%
      \put(218,3259){\makebox(0,0){\strut{}$s$}}%
    }%
    \gplbacktext
    \put(0,0){\includegraphics{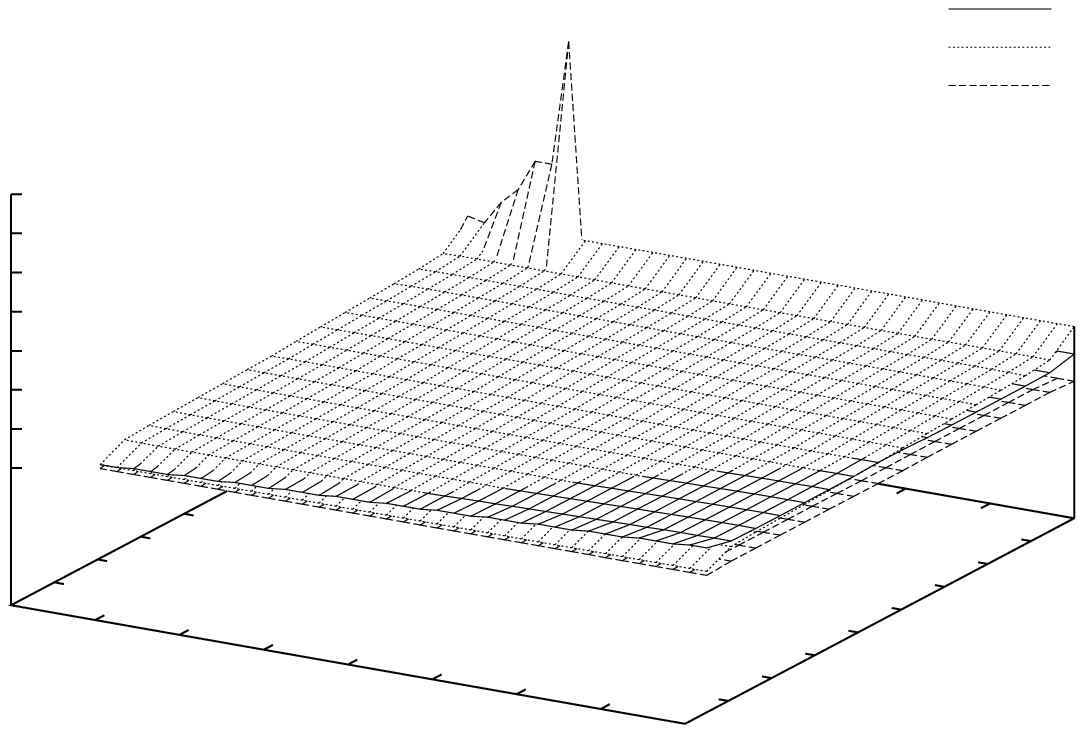}}%
    \gplfronttext
  \end{picture}%
\endgroup